\documentclass[a4,11pt,twoside]{article}
\usepackage{graphicx}
\usepackage{lscape}
\usepackage{enumerate}
\usepackage{latexsym}
\usepackage{amsthm}
\usepackage{color}
\usepackage{framed}

\newtheorem{theorem}{Theorem}

\newtheorem{proposition}{Proposition}
\newtheorem{corollary}{Corollary}
\newtheorem{lemma}{Lemma}

\def\vec#1{\mbox{\boldmath $#1$}}
\usepackage{amsmath,amssymb}

\DeclareMathOperator{\argmax}{argmax}

\DeclareMathOperator{\diag}{{\rm diag}}

\begin{document}
\thispagestyle{plain}

\noindent{\LARGE{\bf Semiparametric Penalized Spline Regression}}

\vskip 3mm
\noindent {\LARGE TAKUMA YOSHIDA$^{1}$ AND KANTA NAITO$^{2}$}

\vskip 1mm

\noindent $^{1}${Graduate School of Science and Engineering,  Shimane University, Matsue, Japan}\\
 $^{2}${Department of Mathematics}, Shimane University, Matsue, Japan

\begin{abstract}
{\it
In this paper, we propose a new semiparametric regression estimator by using a hybrid technique of a parametric approach and a nonparametric penalized spline method. 
The overall shape of the true regression function is captured by the parametric part, while its residual is consistently estimated by the nonparametric part. 
Asymptotic theory for the proposed semiparametric estimator is developed, showing that its behavior is dependent on the asymptotics for the nonparametric penalized spline estimator as well as on the discrepancy between the true regression function and the parametric part.
As a naturally associated application of asymptotics, some criteria for the selection of parametric models are addressed.
Numerical experiments show that the proposed estimator performs better than  the existing kernel-based semiparametric estimator and the fully nonparametric estimator, and that the proposed criteria work well for choosing a reasonable parametric model.}
\end{abstract}

{\small
{\bf Keywords} {
Asymptotic theory;\ Bias reduction;\ $B$-spline;\ Parametric model;\ Penalized spline;\\ \ \ \ \ \ Semiparametric regression}

\vskip 3mm


}
\vspace{3mm}

\section{Introduction}
 
There have been several nonparametric smoothing techniques used in regression problems, such as lowess, kernel smoothing, spline smoothing, wavelet, the series method, and so on.
The nonparametric estimators generally have consistency, which is an advantage of this approach.
Hence, if the nonparametric estimator is used, we can expect that the true regression can be captured as the sample size increases. 
However, because the form of a nonparametric estimator is sometimes complicated, the interpretation of the estimated structure might not be clear. 

On the other hand, in a parametric regression problem with the true regression function controlled by a finite-dimensional parameter vector, the estimated structure is easy to understand, however, the estimator does not have consistency.
Therefore, there are advantages and disadvantages associated with each of these approaches. 
This motivates us to consider a hybrid of parametric and nonparametric methods for the regression problem and we, in fact, introduce a semiparametric regression method so that the estimator has the advantages of both approaches.

The semiparametric method in this paper consists of two steps.
In the first step, we utilize an appropriate parametric estimator.
In the second step, we apply a certain nonparametric smoother to the residual data associated with the parametric estimator in the first step.
The parametric estimator in the first step and the nonparametric smoother in the second step are combined into the proposed semiparametric estimator.

Similar semiparametric approaches for smoothing have been developed by many authors. 
Hjort and Glad (1995) and Naito (2004) discussed similar methods in density estimation literature.
Glad (1998) and Naito (2002) addressed the semiparametric regression method. Martins et al. (2008) introduced general decomposition, including additive and multiplicative corrections in regression.
Recently, Fan et al. (2009) discussed the semiparametric approach in the framework of a generalized linear model. 
Note that the aforementioned works all used kernel smoothing in the second step estimation. 

Our proposal is to utilize the penalized spline method for residual smoothing in the second step. 
This is a typical technique used in nonparametric regression problems with sufficient fitness and appropriate smoothness, which was developed by O'Sullivan (1986) and Eilers and Marx (1996). 
Many of its applications are summarized in Ruppert, et al (2003).
Throughout this paper, the fully nonparametric penalized spline estimator is designated by NPSE, while the semiparametric penalized spline estimator, including the two-step manipulations mentioned above, is denoted by SPSE.
In this paper, the advantages of using the penalized spline method instead of the kernel method are described both theoretically and numerically. 
In particular, we found that the SPSE has better behavior than the semiparametric local linear estimator (SLLE) in simulation.

This paper is organized as follows.
We elaborate on the proposed SPSE in Section 2. 
Section 3 discusses the asymptotic properties of the SPSE, which can be obtained using a combination of the asymptotic results for the parametric estimator and for the NPSE developed by Claeskens et al. (2009). 
The asymptotic bias of the SPSE depends on the initial parametric model utilized in the first step. 
The form of the asymptotic bias suggests a method of choosing the parametric model for the first step. 
A theoretical comparison of SPSE with SLLE is also given in the context of asymptotic bias, which reveals that the use of the penalized spline rather than a kernel smoother in the second step is valid. 
In Section 4, some criteria for parametric model selection will be clarified. 
If a parametric model chosen by the criteria discussed in Section 4 is used as the parametric part of the SPSE, its asymptotic bias will become smaller than that of the NPSE. 
The results of a simulation are reported in Section 5. 
The simulation studies include checking the accuracy of the SPSE and comparing it with the NPSE and the SLLE as regression estimators.
The performance of the parametric model selection discussed in Section 4 is also investigated.
Related discussion and issues for future research are provided in Section 6. 
Proofs for the theoretical results are given in the Appendix.

\section{Semiparametric penalized spline estimator}

Consider the relationship of the dataset $\{(x_i,y_i):i=1,\cdots,n\}$ as the regression model
$$
y_i=f(x_i)+\varepsilon_i,\ \ i=1,\cdots,n,
$$
where the explanatory $x_i$ is generated from density $q(x)$ with its support on $[0,1]$, $f(x)=E[Y|X=x]$ is an unknown regression function, and the errors $\varepsilon_i$ are assumed to be uncorrelated with $E[\varepsilon_i|X_i=x_i]=0$ and $V[\varepsilon_i|X_i=x_i]=\sigma^2(x_i)<\infty$. 
Let $f(x|\vec{\beta}), \vec{\beta}\in B \subseteq \mathbb{R}^M$ be a parametric model. 
We now construct the semiparametric estimator of $f(x)$. 
First we obtain an appropriate estimator $\hat{\vec{\beta}}$ of $\vec{\beta}$ via a suitable method of estimation. 
Then $f(x)$ can be written as
\begin{eqnarray}
f(x)= f(x|\hat{\vec{\beta}})+f(x|\hat{\vec{\beta}})^{\gamma}r_\gamma(x,\hat{\vec{\beta}}), \label{uni}
\end{eqnarray}
where $r_\gamma(x,\vec{\beta})=\{f(x)-f(x|\vec{\beta})\}/f(x|\vec{\beta})^\gamma$ for some $\gamma\in\{0,1\}$. 
When $\gamma=0$, this decomposition becomes $f(x)=f(x|\hat{\vec{\beta}})+ \{f(x)-f(x|\hat{\vec{\beta}})\}$, which is called an additive correction. 
When $\gamma=1$, on the other hand, we have a multiplicative correction $f(x)=f(x|\hat{\vec{\beta}})\{f(x)/f(x|\hat{\vec{\beta}})\}$. 
By using the parameter $\gamma$, we can treat additive and multiplicative corrections systematically (see, Fan et al. (2009)).
In the second step, $r_\gamma(x,\hat{\vec{\beta}})$ is estimated by applying a nonparametric technique to $\{(x_i,\{y_i-f(x_i|\hat{\vec{\beta}})\}/f(x_i|\hat{\vec{\beta}})^\gamma):i=1,\cdots,n\}$.  
The SPSE is obtained as
\begin{eqnarray}
\hat{f}(x,\gamma)=f(x|\hat{\vec{\beta}})+f(x|\hat{\vec{\beta}})^{\gamma}\hat{r}_\gamma(x,\hat{\vec{\beta}}), \label{semiest}
\end{eqnarray}
where $\hat{r}_\gamma(x,\hat{\vec{\beta}})$ is a nonparametric estimator of $r_\gamma(x,\hat{\vec{\beta}})$. 

We adopt the penalized spline to estimate $r_{\gamma}(x,\hat{\vec{\beta}})$.  
Let $\{B_{-p+1}^{[p]}(x),\cdots,B_{K_n}^{[p]}(x)\}$ be a marginal $B$-spline basis of degree $p$ with equally spaced knots $\kappa_k=k/K_n (k=-p+1,\cdots,K_n+p)$. Then we consider the $B$-spline model
$$
\sum_{k=-p+1}^{K_n} B_k^{[p]}(x)b_{k}
$$ 
as an approximation to $r_\gamma(x,\hat{\vec{\beta}})$, where 
$b_{k}$'s are unknown parameters. 
The definition and fundamental properties of the $B$-spline basis are detailed in de Boor (2001). 
Let $\vec{R}_{\gamma}$ be the $n$-vector with $i$th element $\{y_i-f(x_i|\hat{\vec{\beta}})\}/f(x_i|\hat{\vec{\beta}})^\gamma$ and let 
$Z=(B^{[p]}_{-p+j}(x_i))_{ij}$ and $\vec{b}=(b_{-p+1}\ \cdots\ b_{Kn})^\prime$. 
The penalized spline estimator $\hat{\vec{b}}=(\hat{b}_{-p+1}\ \cdots\ \hat{b}_{K_n})^\prime$ of $\vec{b}$ is defined as the minimizer of 
\begin{eqnarray*}
(\vec{R}_\gamma-Z\vec{b})^\prime(\vec{R}_\gamma-Z\vec{b})+\lambda_n \vec{b}^\prime Q_m \vec{b},
\end{eqnarray*}
where $\lambda_n$ is the smoothing parameter and $Q_m$ is the $m$th difference matrix.
The estimator of $r_\gamma(x,\hat{\vec{\beta}})$ is defined as 
\begin{eqnarray}
\hat{r}_\gamma(x,\hat{\vec{\beta}})=\sum_{k=-p+1}^{K_n} B_k^{[p]}(x)\hat{b}_{k}=\vec{B}(x)^\prime (Z^\prime Z+\lambda_n Q_m)^{-1}Z^\prime \vec{R}_{\gamma}, \label{semiPsp}
\end{eqnarray}
where $\vec{B}(x)=(B^{[p]}_{-p+1}(x)\ \cdots\ B^{[p]}_{K_n}(x))^\prime$. 

In Figure 1, an example of the SPSE is drawn. 
In the left panel, the true function $f(x)=\exp[-x^2]\sin(2\pi x)$ and the least square estimator $f(x|\hat{\vec{\beta}})$ of $f(x|\vec{\beta})=\beta_0+\beta_1x+\beta_2x^2+\beta_3x^3$ are shown.  
In the middle panel, the residuals of $f(x|\hat{\vec{\beta}})$ and the penalized spline estimator of $r_0(x,\hat{\vec{\beta}})$ are drawn. 
In the right panel, the true function and the SPSE as given in (\ref{semiest}) are drawn. 
As the interpretation of $\hat{f}(x)$ for this example, the parametric part captures the overall shape of $f(x)$ and the nonparametric part explains details which could not be captured by the $f(x|\hat{\vec{\beta}})$. 
Similarly, we can construct an SPSE with multiplicative correction. 

\begin{figure}
\begin{center}
\includegraphics[width=50mm,height=40mm]{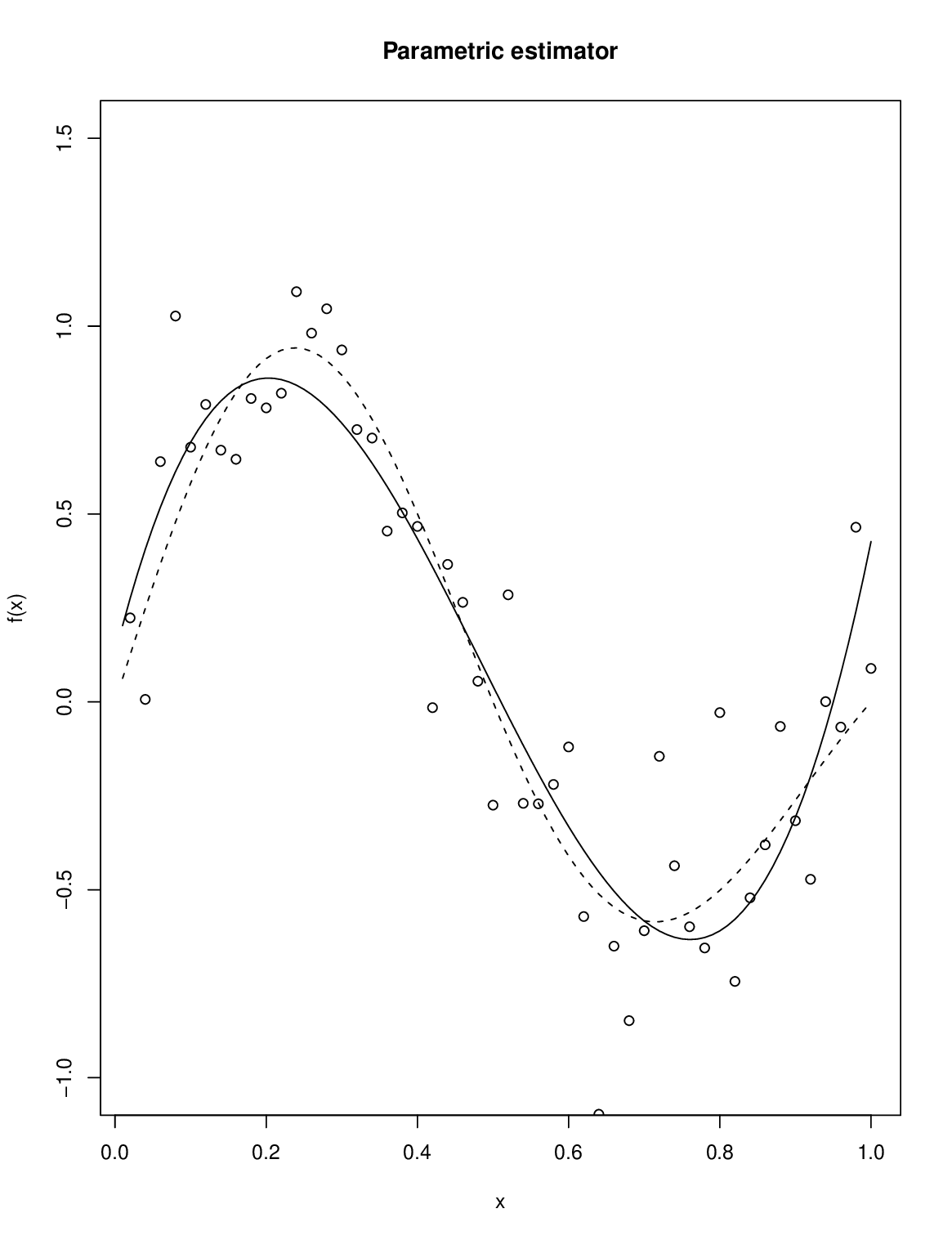}
\includegraphics[width=50mm,height=40mm]{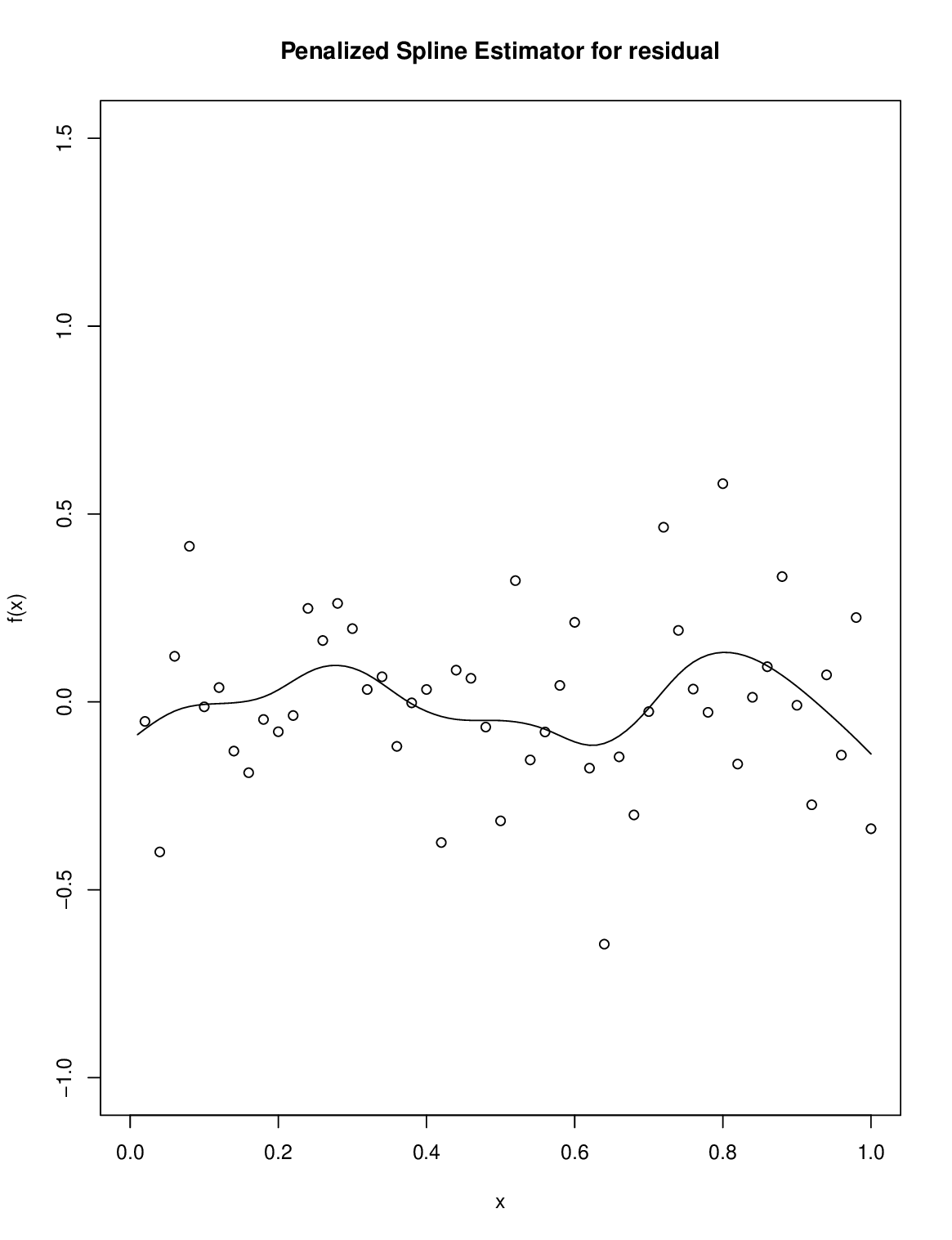}
\includegraphics[width=50mm,height=40mm]{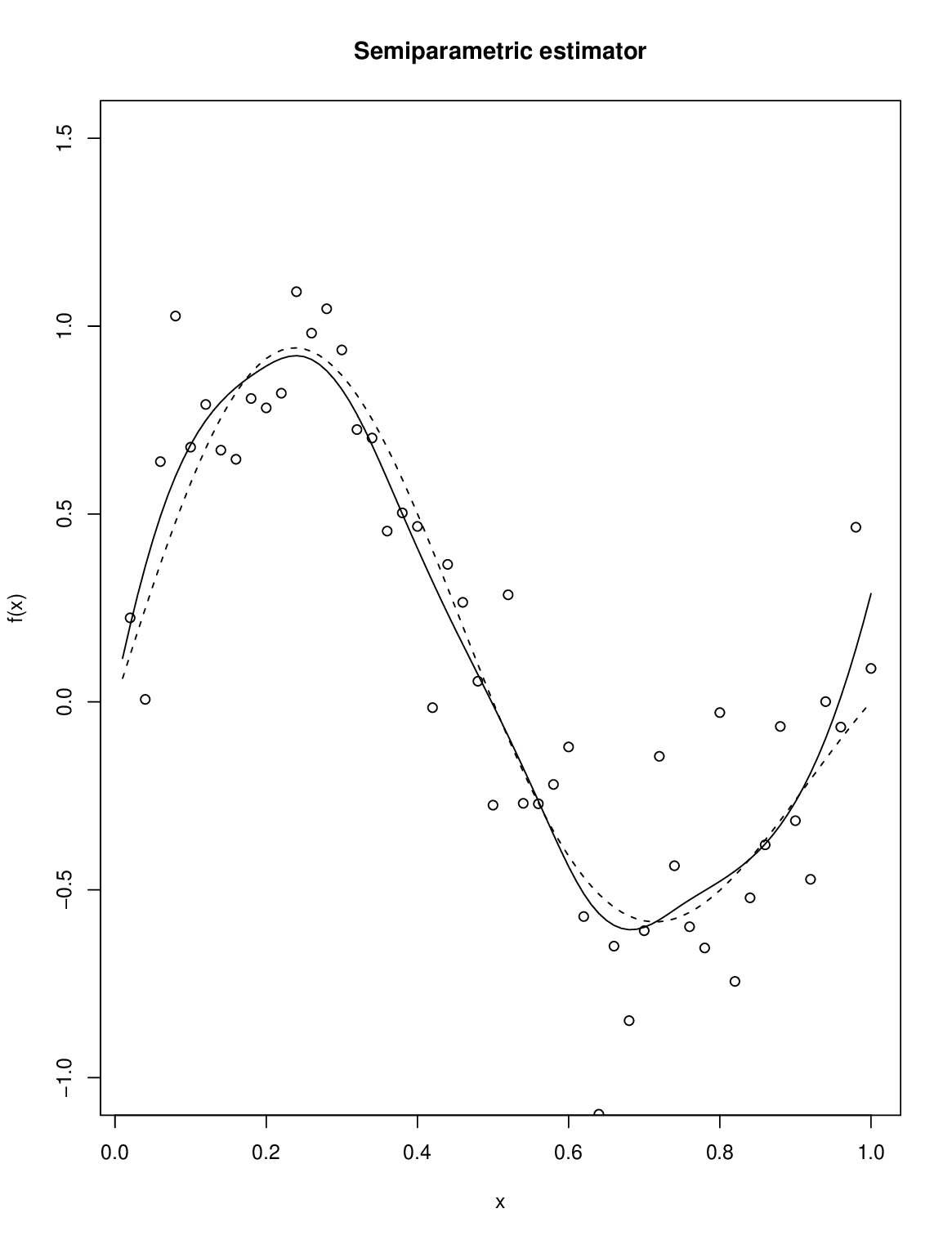}
\end{center}
\caption{Plots for one random sample of true $f(x)$ (dashed) and the parametric estimator $f(x|\hat{\vec{\beta}})$ (solid) in the left panel, the residuals and the penalized spline estimator of $\hat{r}_0(x,\hat{\vec{\beta}})$ (solid) in the middle panel, and the true $f(x)$ (dashed) and the SPSE $\hat{f}(x,0)$ (solid) in the right panel.}
\end{figure}


\section{Asymptotic Result}

Asymptotics for the NPSE were developed by Claeskens et al. (2009). 
By using their results, we show the asymptotic bias and variance, and asymptotic distribution of the SPSE.
We now give some assumptions regarding the asymptotics of the SPSE. 
\\

{\noindent \bf Assumptions}
\begin{enumerate}
\item There exists $a>0$ such that $a<f(x|\vec{\beta})$ for all $x\in [0,1]$, $\vec{\beta}\in B$. 
\item $\sup_{z\in[0,1]}\{q(z)\}<\infty$.
\item $|\partial f(x|\vec{\beta})/\partial \beta_i|<\infty,\ {\rm for}\ x\in [0,1]$, $\vec{\beta}\in B$, $i=1,\cdots,m.$
\item $|\partial^2 f(x|\vec{\beta})/\partial \beta_i\partial\beta_j|<\infty, \ {\rm for}\ x\in [0,1]$, $\vec{\beta}\in B$, $i,j=1,\cdots,m.$
\item $|d^i f(x)/d x^i|<\infty,\ {\rm for}\ x\in [0,1]$, $i=1,\cdots,p+1$.
\item $K_n=o(n^{1/2})$ and $\lambda_n=o(nK_n^{-1})$.
\end{enumerate}

\noindent
Define the $(K_{n}+p)\times(K_{n}+p)$ matrix $G(q)=(g_{ij})_{ij}$, where 
$$
g_{ij}=\int_0^1 B^{[p]}_{-p+i}(u)B^{[p]}_{-p+j}(u)q(u)du
$$
and the $(K_{n}+p)\times(K_{n}+p)$ matrix $G(\sigma,\beta,\gamma,q)=(g_{\sigma,ij})_{ij}$, where 
$$
g_{\sigma,ij}=\int_0^1 B^{[p]}_{-p+i}(u)B^{[p]}_{-p+j}(u)\frac{\sigma^2(u)q(u)}{f(u|\vec{\beta})^{2\gamma}}du.
$$
Let $\vec{b}^*(\vec{\beta},\gamma)$ be a best $L_\infty$ approximation to $(f(x)-f(x|\vec{\beta}))/f(x|\vec{\beta})^\gamma$. 
This means that $\vec{b}^*(\vec{\beta},\gamma)$ satisfies
$$
\sup_{x\in(0,1)}\left|\frac{f(x)-f(x|\vec{\beta})}{f(x|\vec{\beta})^\gamma}+b_{a1}(x|\vec{\beta},\gamma)-\vec{B}(x)^\prime \vec{b}^*(\vec{\beta},\gamma)\right|=o(K_n^{-(p+1)}),
$$
where 
$$
b_{a1}(x|\vec{\beta},\gamma)=-\left(\frac{f(x)-f(x|\vec{\beta})}{f(x|\vec{\beta})^\gamma}\right)^{(p+1)}\frac{1}{K_n^{p+1} (p+1)!}\sum_{j=1}^{K_n}I(\kappa_{j-1}\leq x<\kappa_j)B_{p+1}\left(\frac{x-\kappa_{j-1}}{K_n^{-1}}\right),
$$
$I(a<x<b)$ is the indicator function of the interval $(a,b)$ and $B_p(x)$ is  the $p$th Bernoulli polynomial. 

We now discuss a condition of the parametric estimator. 
Let $F$ be the true distribution of $(X,Y)$ and let $F_n$ be the corresponding empirical distribution. 
The estimator $\hat{\vec{\beta}}$ of $\vec{\beta}$ is defined as the functional form $\hat{\vec{\beta}}=T(F_n)$, where $T(\cdot)$ is a real valued function defined on the set of all distributions. 
We can then see that $\lim_{n\rightarrow \infty} \hat{\vec{\beta}}\rightarrow \vec{\beta}_0$, where $\vec{\beta}_0=T(F)$ is defined as the optimizer of some distance measure $\rho$. 
We assume that $f(x|\vec{\beta}_0)$ is the best approximation of $f(x)$.
By the definition of $\hat{\vec{\beta}}$, $\hat{\vec{\beta}}-\vec{\beta}_0$ can be expressed as
\begin{eqnarray}
\hat{\vec{\beta}}-\vec{\beta}_0=\frac{1}{n}\sum_{i=1}^n I(X_i,Y_i)+\frac{d}{n}+\delta_n, \label{ex}
\end{eqnarray}
where $I(X_i,Y_i)$ is the influence function defined as 
$$
I(X,Y)=\lim_{\varepsilon\rightarrow 0}\left\{\frac{T((1-\varepsilon)F+\varepsilon \delta(X,Y))-T(F)}{\varepsilon}\right\}
$$
with $E[I(X_i,Y_i)]=0$ and finite covariance matrix, the delta function $\delta(X,Y)$ has probability 1 at a point $(X,Y)$, and $d$ is the bias of $\hat{\vec{\beta}}$. 
The remaining term $\delta_n$ has mean $O(n^{-2})$ for each component. 
 
We investigate the asymptotic property of $\hat{f}(x,\gamma)$ by a two-step procedure for clarity. 
First we derive the asymptotic expectation and variance of $\hat{f}_0(x,\gamma)=f(x|\vec{\beta}_0)+f(x|\vec{\beta}_0)^\gamma \hat{r}_\gamma(x,\vec{\beta}_0)$. 
Here, $\hat{r}_\gamma(x,\vec{\beta}_0)$ is the penalized spline smoother of $r_\gamma(x,\vec{\beta}_0)$.
Second, we show that the difference between $\hat{f}(x,\gamma)$ and $\hat{f}_0(x,\gamma)$ vanishes asymptotically.  
Since $\vec{\beta}_0$ is no longer stochastic, the asymptotic property of $\hat{f}_0(x,\gamma)$ is dependent only on the nonparametric penalized spline estimator of $r_\gamma(x,\vec{\beta}_0)$. 
Hence we obtain 
\begin{eqnarray*}
E[\hat{f}_0(x,\gamma)|\vec{X}_n]&=&f(x|\vec{\beta}_0)+f(x|\vec{\beta}_0)^\gamma E[\hat{r}_\gamma(x,\vec{\beta}_0)],\\
V[\hat{f}_0(x,\gamma)|\vec{X}_n]&=&f(x|\vec{\beta}_0)^{2\gamma} V[\hat{r}_\gamma(x,\vec{\beta}_0)].
\end{eqnarray*}
Here for a random variable $U_n$, $E[U_n|\vec{X}_n]$ and $V[U_n|\vec{X}_n]$ are the conditional expectation and variance of $U_n$ given $(X_1,\cdots,X_n)=(x_1,\cdots,x_n)$.
The asymptotic property of $\hat{r}_\gamma(x,\vec{\beta}_0)$ can be directly obtained by using Theorem 2 (a) of Claeskens et al. (2009).

\begin{proposition}\label{fix}
Let $f\in C^{p+1}, f(\cdot|\vec{\beta})\in C^{p+1}$. 
Then, under the Assumptions, for a fixed $x\in(0,1)$, 
\begin{eqnarray*}
E[\hat{f}_0(x,\gamma)|\vec{X}_n]&=&f(x)+b_a(x|\vec{\beta}_0,\gamma)+b_\lambda(x|\vec{\beta}_0,\gamma)+o_P(K_n^{-(p+1)})+o_P(\lambda_nK_nn^{-1}),\\
V[\hat{f}_0(x,\gamma)|\vec{X}_n]&=&\frac{f(x|\vec{\beta}_0)^{2\gamma}}{n}\vec{B}(x)^\prime G(q)^{-1}G(\sigma,\beta_0,\gamma,q) G(q)^{-1}\vec{B}(x)+o_P(K_nn^{-1}),
\end{eqnarray*}
where 
\begin{eqnarray*}
b_{a}(x|\vec{\beta}_0,\gamma)&=&-\frac{f(x|\vec{\beta}_0)r_\gamma^{(p+1)}(x|\vec{\beta}_0)}{K_n^{p+1} (p+1)!}\sum_{j=1}^{K_n}I(\kappa_{j-1}\leq x<\kappa_j)B_{p+1}\left(\frac{x-\kappa_{j-1}}{K_n^{-1}}\right),\\
b_{\lambda }(x|\vec{\beta}_0,\gamma)&=&-\frac{\lambda_n}{n}f(x|\vec{\beta}_0)^\gamma\vec{B}(x)^\prime G(q)^{-1}Q_m \vec{b}^*(\vec{\beta}_0,\gamma).
\end{eqnarray*}

\end{proposition}

We now give the asymptotic result for $\hat{f}(x,\gamma)$. 
By using (\ref{ex}), $f(x|\hat{\vec{\beta}})$ and $\hat{r}_\gamma(x,\hat{\vec{\beta}})$ are expanded about $f(x|\vec{\beta}_0)$ and $\hat{r}_\gamma(x,\vec{\beta}_0)$, respectively. 
From the details of the proof in the Appendix, we find that the asymptotic expectation and variance of $\hat{f}(x,\gamma)$ are dominated by those of $\hat{f}_0(x,\gamma)$ and we obtain the following theorem. 

\begin{theorem}\label{est}
Let $f\in C^{p+1}, f(\cdot|\vec{\beta}_0)\in C^{p+1}$. 
Then under the Assumptions, for a fixed $x\in(0,1)$, 
\begin{eqnarray*}
E[\hat{f}(x,\gamma)|\vec{X}_n]&=&f(x)+b_a(x|\vec{\beta}_0,\gamma)+b_\lambda(x|\vec{\beta}_0,\gamma)\\
&&+ O_P(n^{-1})+o_P(K_n^{-(p+1)})+o_P(\lambda_nK_nn^{-1}),\\
V[\hat{f}(x,\gamma)|\vec{X}_n]&=&\frac{f(x|\vec{\beta}_0)^{2\gamma}}{n}\vec{B}(x)^\prime G(q)^{-1}G(\sigma,\beta_0,\gamma,q) G(q)^{-1}\vec{B}(x)+o_P(K_nn^{-1}),
\end{eqnarray*}
where $b_{a}(x|\vec{\beta}_0,\gamma)$ and $b_{\lambda }(x|\vec{\beta}_0,\gamma)$ are those given in Proposition \ref{fix}. 
\end{theorem}  

Theorem \ref{est} and Lyapunov's theorem yield the asymptotic distribution of the SPSE.

\begin{theorem}\label{norm}
Suppose that $E[|\varepsilon_i|^{2+\delta}|X_i=x_i]<C$ for some $\delta\geq 2$ and the Assumptions are satisfied. 
Then, using $K_n=O(n^{1/(2p+1)})$ and $\lambda_n=O(n^{p/(2p+1)})$, 
\begin{eqnarray*}
\frac{\hat{f}(x,\gamma)-f(x)-b_a(x|\vec{\beta}_0,\gamma)-b_\lambda(x|\vec{\beta}_0,\gamma)}{\sqrt{V[\hat{f}(x,\gamma)|\vec{X}_n]}}
\xrightarrow {D}N(0,1),
\end{eqnarray*}
where $b_{a}(x|\vec{\beta}_0,\gamma)$ and $b_{\lambda }(x|\vec{\beta}_0,\gamma)$ are those given in Proposition \ref{fix}.
\end{theorem}

If $\lambda_n=0$, we obtain the semiparametric regression spline estimator from (\ref{semiest}). 
Thus, it is clear that the asymptotic result of the semiparametric regression spline is contained in Theorems \ref{est} and \ref{norm}. 
These are obtained from one parametric model. 
If we choose a polynomial model as $f(x|\vec{\beta})$, we obtain the following Corollary. 

\begin{corollary}\label{pol}
Let $f_q(x|\vec{\beta}_q) (q\leq p)$ be the $q$th polynomial model. 
Then, under $\lambda_n=0$ and $\gamma=0$, or $\lambda_n>0$ and $\gamma=0$,  using $p=1$, $Q_2$ and equidistant knots, the SPSE is the same as the NPSE. 
\end{corollary}

\noindent{\bf Remark 1}
\quad From Theorem \ref{norm}, as the advanced analysis, we can construct the asymptotic pointwise confidence interval of $f(x)$ by estimating the variance of the error. 
\\ 

\noindent{\bf Remark 2}
\quad Theorems \ref{est} and \ref{norm} can be applied for $\gamma\in\{0,1\}$. 
When $\gamma=0$, the results become those for additive correction. 
When $\gamma=1$, $b_a(x|\vec{\beta}_0,1)$ and the variance agrees with that of the estimator for multiplicative correction. 
In $b_\lambda(x|\vec{\beta}_0,1)$, it is understood that $\vec{b}^*(\vec{\beta}_0,1)$ is a best $L_\infty$ approximation of $f(x)/f(x|\vec{\beta}_0)-1$. 
Therefore, $\vec{b}^*(\vec{\beta}_0,1)$ can be written as $\vec{b}^*(\vec{\beta}_0,1)=\vec{b}^{*}-\vec{1}$, where $\vec{b}^{*}$ is a best $L_\infty$ approximation of $f(x)/f(x|\vec{\beta}_0)$ and $\vec{1}$ is a $(K_n+p)$ vector with all components equal to 1. 
In conclusion, $b_\lambda(x|\vec{\beta}_0,1)$ can be written as
$$
b_{\lambda }(x|\vec{\beta}_0,1)=-\frac{\lambda_n}{n}f(x|\vec{\beta}_0)^\gamma\vec{B}(x)^\prime G(q)^{-1}Q_m \vec{b}^*
$$
because all components of $Q_m\vec{1}$ have vanished.   
\\

\noindent{\bf Remark 3}
\quad When $f(x)=f(x|\vec{\beta}_0)$ is assumed, we obtain $b_a(x|\vec{\beta}_0,\gamma)=0$ and $b_\lambda(x|\vec{\beta}_0,\gamma)=0$ by choosing $\vec{b}^*(\gamma,\vec{\beta}_0)=\vec{0}$ as a best $L_\infty$ approximation of 0. 
For $\gamma=1$, in particular, $b_a(x|\vec{\beta}_0,1)=0$ and $b_\lambda(x|\vec{\beta}_0,1)=0$ both hold even in cases where $f(x)=cf(x|\vec{\beta}_0)$ with any constant $c\not=0$.
\\

\noindent{\bf Remark 4}
\quad If we use the local $p$th polynomial technique in the second step estimation, we obtain the asymptotic bias $b_\ell(x|\vec{\beta}_0)$ as 
{\small
\begin{eqnarray*}
b_\ell(x|\vec{\beta}_0,\gamma)
=
\left\{
\begin{array}{cl}
-h_n^{p+1}\displaystyle\frac{f(x|\vec{\beta}_0)r_\gamma^{(p+1)}(x|\vec{\beta}_0)}{(p+1)!} 
\int_{\mathbb{R}} z^{p+1}H_{p}(z)dz,& p\ :\ {\rm odd},\\
-h_n^{p+2}f(x|\vec{\beta}_0)\left\{\displaystyle\frac{r_\gamma^{(p+2)}(x|\vec{\beta}_0)}{(p+2)!} +\displaystyle\frac{r_\gamma^{(p+1)}(x|\vec{\beta}_0)q^\prime(x)}{(p+1)!q(x)}\right\}\displaystyle\int_{\mathbb{R}} z^{p+2}H_{p}(z)dz,& p\ :\ {\rm even},
\end{array}
\right.
\end{eqnarray*}
}
where $h_n$ is bandwidth and $H_{p}(z)$ is the $p$th order kernel function. 
If $K_n^{-1}$ and $h_n$ are equal and $p$ is odd, the difference between $b_a(x|\vec{\beta}_0)$ and $b_\ell(x|\vec{\beta}_0)$ is only 
\begin{eqnarray}
\sum_{j=1}^{K_n}I(\kappa_{j-1}\leq x<\kappa_j)B_{p+1}\left(\frac{x-\kappa_{j-1}}{K_n^{-1}}\right)\ \ {\rm and}\ \  \int_{\mathbb{R}} z^{p+1}H_{p}(z)dz. \label{const}
\end{eqnarray} 
If we can calculate (\ref{const}), we would be able to compare the bias of the SPSE with that of the semiparametric local polynomial kernel estimator.
As an example, when $p=1$, it is easy to show that $B_2(x)=x^2-x+1/6<1/5$ for $x\in[0,1]$, while we have $\int_{\mathbb{R}} z^2H_G(z)dz=1$ for the Gaussian kernel $H_{G}(z)$ and $\int_{\mathbb{R}} z^2H_E(z)dz=1/5$ for the Epanechnikov kernel $H_E(z)$. 
Therefore $b_a(x|\vec{\beta}_0)$ is smaller than $b_\ell(x|\vec{\beta}_0)$ in this situation, which reveals that the SPSE is superior than the SLLE.


\section{Parametric model selection}

In this section, we describe how to choose a parametric model.
From Remark 3, if the true regression function satisfies $f\in \{f(\cdot|\vec{\beta})|\vec{\beta}\in B\subseteq \mathbb{R}^M\}$, the bias of the SPSE is reduced. 
Hence we determine the initial parametric model in a bias reduction context. 
Specifically, our purpose is to choose a parametric model such that the asymptotic bias of the SPSE becomes smaller than that of the NPSE: 
\begin{eqnarray}
|b_a(x|\vec{\beta}_0,\gamma)|<|b_a(x)|\ \ \ {\rm and}\ \ \ |b_{\lambda}(x|\vec{\beta}_0,\gamma)|<|b_{\lambda}(x)|,\  {\rm for \ all} \ x\in(0,1), \label{cond1}
\end{eqnarray}
where $b_a(x)$ and $b_{\lambda}(x)$ are the asymptotic biases of the NPSE. 
If $f(x|\vec{\beta})$ is constant, $b_a(x|\vec{\beta}_0,\gamma)$ and $b_\lambda(x|\vec{\beta}_0,\gamma)$ are equivalent to $b_a(x)$ and $b_{\lambda}(x)$, respectively.
When the same $K_n$ and $\lambda_n$ are used in both the SPSE and the NPSE, 
(\ref{cond1}) can be rewritten as $L_a(x,\gamma)>0$ and $L_{\lambda}(x,\gamma)>0$ for all $x\in(0,1)$, where
\begin{eqnarray*}
L_a(x,\gamma)=|f^{(p+1)}(x)|-\left|f(x|\vec{\beta}_0)^\gamma\left(\frac{f(x)-f(x|\vec{\beta}_0)}{f(x|\vec{\beta}_0)^\gamma}\right)^{(p+1)}\right|
\end{eqnarray*}
and 
\begin{eqnarray*}
L_\lambda(x,\gamma)=|\vec{B}(x)^\prime G(q)^{-1}Q_m\vec{b}_f^*|-|f(x|\vec{\beta}_0)^\gamma\vec{B}(x)^\prime G(q)^{-1}Q_m\vec{b}^*(\vec{\beta}_0,\gamma)|,
\end{eqnarray*}
where $\vec{b}_f^*$ is a best $L_\infty$ approximation to $f(x)$.
As a pilot estimator of $f$ and its $(p+1)$th derivative, we can use the local polynomial estimator $\hat{f}$ with degree $p+2$. 
Then the estimator of $L_a(x,\gamma)$ and $L_{\lambda}(x,\gamma)$ can be obtained as 
\begin{eqnarray*}
\hat{L}_a(x,\gamma)=|\hat{f}^{(p+1)}(x)|-\left|f(x|\hat{\vec{\beta}})^\gamma\left(\frac{\hat{f}(x)-f(x|\hat{\vec{\beta}})}{f(x|\hat{\vec{\beta}})^\gamma}\right)^{(p+1)}\right|
\end{eqnarray*} 
and by using empirical form, 
\begin{eqnarray*}
\hat{L}_\lambda(x,\gamma)=|\vec{B}(x)^\prime\Lambda^{-1}Q_m(Z^\prime Z)^{-1}Z^\prime \hat{\vec{f}}|
-|f(x|\hat{\vec{\beta}})^\gamma\vec{B}(x)^\prime\Lambda^{-1}Q_m(Z^\prime Z)^{-1}Z^\prime \hat{\vec{r}}_\gamma|,
\end{eqnarray*}
where $\hat{\vec{f}}=(\hat{f}(x_1)\ \cdots\ \hat{f}(x_n))^\prime$ and $\hat{\vec{r}}_\gamma$ is an $n$-vector with $i$th component $\{\hat{f}(x_i)-f(x_i|\hat{\vec{\beta}})\}/f(x_i|\hat{\vec{\beta}})^\gamma$.
Here, we use the fact that
\begin{eqnarray*}
\lambda_nf(x|\hat{\vec{\beta}})^\gamma\vec{B}(x)^\prime\Lambda^{-1}Q_m(Z^\prime Z)^{-1}Z^\prime \hat{\vec{r}}_\gamma =b_{\lambda}(x|\vec{\beta}_0,\gamma)+o_P(\lambda_nK_nn^{-1}),
\end{eqnarray*}
which is detailed in the proof of Theorem 2 (a) of Claeskens et al. (2009). 
We choose one parametric model by relative evaluation. 
Let 
\begin{eqnarray*}
C_{a\cap \lambda}(f(\cdot|\vec{\beta}))=\#\left\{z_j\in(0,1)\Bigl| \hat{L}_a(z_j,\gamma)>0, \hat{L}_{\lambda}(z_j,\gamma)>0, j=1,\cdots,J\right\},
\end{eqnarray*}
for a given parametric model $f(\cdot|\vec{\beta})$ and some finite grid points $\{z_j\}_{1}^{J}$ on $(0,1)$. 
Here, for a set $A$, $\# A$ is the cardinality of $A$.
After preparing a class of candidate parametric models $\{f_k=f_k(\cdot|\vec{\beta}_k);k=1,\cdots,K\}$, 
we choose a parametric model satisfying
\begin{eqnarray}
f(x|\vec{\beta})=\underset{f_k}{\argmax} \left\{C_{a\cap \lambda}(f(\cdot|\vec{\beta}_k))\right\}. \label{modelsel}
\end{eqnarray}
In summary, for each parametric model $f_k$, we calculate $\hat{L}_a$, $\hat{L}_{\lambda}$ and $C_{a\cap \lambda}(f(\cdot|\vec{\beta}_k))$. 
By using the parametric model which satisfies (\ref{modelsel}), we construct the SPSE. 
If we can choose a good parametric model and a good $\hat{\vec{\beta}}$, the SPSE will have better behavior than the NPSE. 
\\

\noindent{\bf Remark 5} 
\quad When we construct the semiparametric regression spline estimator (SPSE with $\lambda_{n}=0$), we obtain $b_\lambda(x|\vec{\beta}_0,\gamma)\equiv 0$. 
Therefore, $C_{a\cap\lambda}$ depends only on $L_a(x,\gamma)$. 
\\

\noindent{\bf Remark 6}
\quad We see that the bias term $b_a(x|\vec{\beta}_0,\gamma)$ appears due to the use of the $B$-spline model. 
On the other hand, $b_{\lambda}(x|\vec{\beta}_0,\gamma)$ arises from the penalty component. 
If we use the regression spline, $b_{\lambda}(x|\vec{\beta}_0,\gamma)$ vanishes and the bias of the estimator becomes less than that of the penalized spline estimator. 
However, the regression spline often provides overfitting. 
Thus, we use the penalized method for obtaining a smooth curve.  
If $\lambda_n>0$, a certain amount of smoothness in the estimator is assured. 
However, $b_{\lambda}(x|\vec{\beta}_0, \gamma)$ may grow too large because of the influence of the parametric model. 
Therefore under $\lambda_n>0$, we suggest choosing $f(x|\vec{\beta})$ such that $b_{\lambda}(x|\vec{\beta}_0,\gamma)$ becomes less than $b_\lambda(x)$. 
Hence, together with $L_a(x,\gamma)$, the parametric model chosen by $C_{a\cap\lambda}$ appears to bring fitness and smoothness to the SPSE.

\section{Simulation}

In this section, we examine the results of a numerical study to confirm the effects of the SPSE on a finite sample. 
We choose a parametric model by the criteria discussed in Section 4. 
We also compare the performance of the SPSE to those of the NPSE, the SLLE and the fully nonparametric local linear estimator (NLLE). 
In all situations, we utilize the linear and cubic splines and the second difference penalty for the second step nonparametric estimation. 
The SPSEs with linear and cubic splines are designated as SPSE1 and SPSE3, respectively. 
NPSE1 and NPSE3 are labeled similarly. 
The number of knots and the smoothing parameter are determined by GCV. 
The design points $\{x_i\}_1^n$ are drawn from a uniform density on $[0,1]$ and the errors $\{\varepsilon_i\}_1^n$ are generated from the normal with mean 0 and variance $\sigma^2(x_i)$. 
Let 
\begin{eqnarray*}
C_a&=&C_a(f(\cdot|\vec{\beta}))=\#\left\{z_j\in(0,1)\Bigl| \hat{L}_a(z_j,\gamma)>0, j=1\cdots,J\right\},\\
C_\lambda&=&C_\lambda(f(\cdot|\vec{\beta}))=\#\left\{z_j\in(0,1)\Bigl| \hat{L}_{\lambda}(z_j,\gamma)>0, j=1,\cdots,J\right\},\\
C_{a\cap \lambda}&=&C_{a\cap \lambda}(f(\cdot|\vec{\beta}))=\#\left\{z_j\in(0,1)\Bigl| \hat{L}_a(z_j,\gamma)>0, \hat{L}_{\lambda}(z_j,\gamma)>0, j=1,\cdots,J\right\},
\end{eqnarray*}
where $z_j=j/J, J=100$.
We prepare a class of candidate parametric models $\{f_k=f_k(\cdot|\vec{\beta}_k)|k=1,\cdots,K\}$.
For each $f_k$, we calculate $C_a$, $C_\lambda$ and $C_{a\cap\lambda}$. 
We use a number of repetitions $R=1000$. 
For each iteration, we pick up $f_k$ from candidate models which maximize $C_a$. 
The same manipulation is implemented for $C_\lambda$ and $C_{a\cap\lambda}$. 
Finally we count the number of times that $f_k$ is picked up during the iterations. 
For comparison, we also show the model selection by using the AIC and the Takeuchi information criterion (TIC) detailed in Konishi and Kitagawa (2008).  

Let 
\begin{eqnarray*}
B_j=\frac{1}{R}\sum_{r=1}^{R}\hat{f}_r(z_j)-f(z_j),\ \ \ 
{\rm V}_j=\frac{1}{R}\sum_{r=1}^{R}\left\{\hat{f}_r(z_j)-\frac{1}{R}\sum_{r=1}^{R}\hat{f}_r(z_j)\right\}^2,
\end{eqnarray*}
where $\hat{f}_r(z_j)$ is the estimator for the $r$th repetition. 
Let 
${\rm ISB}=100^{-1}\sum_{j=1}^{100}B_j^2$, 
${\rm V}=100^{-1}\sum_{j=1}^{100}{\rm V}_j$ and ${\rm MISE}={\rm ISB}+{\rm V}$  be the estimates of integrated squared bias, integrated variance and mean integrated squared error of $\hat{f}$, respectively. 
For comparison, the ISB, V and MISE of the SLLE and the NLLE were also calculated. 
In the SLLE and the NLLE, we used the Gaussian kernel and its bandwidth $h_n$ was obtained by the direct plug-in approach (Ruppert et al. (1995)). 
\\

\noindent{\bf Example 1}
\quad The true function is $f(x)=2+\sin(2\pi x)$. 
We use three different specified parametric models: 
\begin{eqnarray*}
f(x|\vec{\beta})=
\left\{
\begin{array}{ll}
\beta_0+\beta_1\sin(2\pi x),&f_1={\it sin},\\
\beta_0+\beta_1x, & f_2={\it poly1},\\
\beta_0+\beta_1x+\beta_2^2+\beta_3x^3,&f_3={\it poly3}.
\end{array}
\right.
\end{eqnarray*}
The true curve can be approximated by {\it sin}. 
The curve {\it poly1} is a rough model and {\it poly3} is close to the true $f$. 
The variance of the error is $\sigma^2(x)=(0.5)^2$ and the sample size is $n=25$. 
The coefficients of the covariate are estimated by the maximum likelihood method for each model. 
This set-up is similar to that used by Glad (1998).

\begin{table}
\begin{center}
\caption{The results of parametric model selection in Example 1.}
\begin{tabular}{c|c|ccc|ccc|cc}
\hline
\multicolumn{2}{c}{$n=25$}&\multicolumn{3}{|c|}{SPSE1}&\multicolumn{3}{c|}{SPSE3}&\multicolumn{2}{c}{IC}\\
\cline{3-8}
\hline
method&model&$C_a$&$C_\lambda$&$C_{a\cap\lambda}$&$C_a$&$C_\lambda$&$C_{a\cap\lambda}$&AIC&TIC\\
\hline
&{\it sin}& 
1000&901&1000&
1000&1000&1000&
850&938
\\
$\gamma=0$&{\it poly1}& 
0&0&0&
0&0&0&
0&0
\\
&{\it poly3}& 
0&99&0&
0&0&0&
150&62
\\
\hline
&{\it sin}& 
997&917&974&
997&837&953&
850&938
\\
$\gamma=1$&{\it poly1}& 
0&34&3&
0&77&4&
0&0
\\
&{\it poly3}& 
1&33&20&
3&86&43&
150&62
\\
\hline
\end{tabular}
\end{center}
\end{table}

Table 1 includes the number of times that each parametric model $f_k$ was chosen based on each criterion. 
In $C_a$, $C_\lambda$ and $C_{a\cap \lambda}$, {\it sin} was selected in almost all iterations. 
This result is desirable because {\it sin} coincides with the true function $f$. 
We also observe that the AIC and the TIC often choose {\it sin}. 
When the number of times {\it sin} is chosen is taken into consideration, it seems that $C_{a\cap \lambda}$ is a better selector than the AIC and the TIC. 

Results for ISB, V and MISE of the SPSE and the NPSE are given in Table 2. 
The SPSE with {\it sin} succeeds in regards to bias reduction even with a small sample size, and variance and MISE of the SPSE are also smaller than those of the NPSE. 
In additive correction, the result of SPSE1 with {\it poly1} is exactly the same as that of the NPSE (see Corollary 1). 
If we use {\it poly3}, MISE of the SPSE is smaller than that of the NPSE, although the squared bias is somewhat larger in multiplicative correction. 
In both ${\rm ISB}$, V and MISE, the values of the SPSE are smaller than those of the SLLE.
We implemented the same method of analysis for the case $n=200$. 
The ${\rm ISB}$, V and MISE of the SPSE and those of the NPSE were almost the same, although these are not shown in this paper. 
\\

\begin{table}
\begin{center}
\caption{Results of integrated squared bias, variance and mean integrated squared bias of Example 1. All entries for ISB,V and MISE are $10^3$ times their actual values.}
\scalebox{0.9}[0.9]{
\begin{tabular}{c|c|ccc|ccc|ccc}
\hline
\multicolumn{2}{c|}{$n=25$}&\multicolumn{3}{c|}{SPSE1}&\multicolumn{3}{c}{SPSE3}&\multicolumn{3}{|c}{SLLE}\\
\cline{3-8}
\hline
method&model&${\rm ISB}$&V&MISE&${\rm ISB}$&V&MISE&${\rm ISB}$&V&MISE\\
\hline
&{\it sin}& 
0.009&8.308&8.318&
0.009&7.907&7.917&
0.029&9.032&9.061
\\
$\gamma=0$&{\it poly1}& 
1.450&12.111&13.562&
1.110&10.056&11.166&
2.370&14.105&16.476
\\
&{\it poly3}& 
1.250&10.949&12.199&
0.873&9.636&10.510&
2.071&15.825&17.898
\\
\hline
&{\it sin}& 
0.011&8.394&8.405&
0.010&8.292&8.302&
0.026&10.708&10.734
\\
$\gamma=1$&{\it poly1}& 
1.571&12.322&13.893&
1.565&12.212&13.777&
2.357&13.860& 16.217
\\
&{\it poly3}& 
2.016&11.198&13.215&
1.016&10.198&11.215&
2.942&12.472&15.415
\\
\hline
\hline
\multicolumn{2}{c}{$n=25$}&\multicolumn{3}{|c|}{NPSE1}&\multicolumn{3}{|c}{NPSE3}&\multicolumn{3}{|c}{NLLE}\\
\hline
\multicolumn{2}{c|}{Fully nonparametric}&ISB&V&MISE&ISB&V&MISE&ISB&V&MISE\\
\cline{3-11}
\multicolumn{2}{c|}{method}
&
1.450&12.111&13.562&
1.108&11.030&12.138&
2.370&14.105&16.476
\\
\hline
\end{tabular}
}
\end{center}
\end{table}

\noindent{\bf Example 2}
\quad The same true function $f$ used in Example 1 is adopted and the sample size is $n=25$.
A class of initial parametric models is chosen, consisting of $q$th degree polynomials ranging from $q=1$ to $6$ and designated as  {\it poly1}, ..., {\it poly6}, respectively, and $\sigma^2=1$. 
This parametric model clearly does not contain the true $f$ and the estimator becomes unstable because the variance of error is relatively large. 

\begin{table}
\begin{center}
\caption{The results of parametric model selection in Example 2.}
\begin{tabular}{c|c|ccc|ccc|cc}
\hline
\multicolumn{2}{c}{$n=25$}&\multicolumn{3}{|c|}{SPSE1}&\multicolumn{3}{c|}{SPSE3}&\multicolumn{2}{c}{IC}\\
\cline{3-8}
\hline
method&model&$C_a$&$C_\lambda$&$C_{a\cap\lambda}$&$C_a$&$C_\lambda$&$C_{a\cap\lambda}$&AIC&TIC\\
\hline
&{\it poly1}& 
0&0&0&
0&0&0&
0&0
\\
&{\it poly2}& 
0&49&30&
0&8&0&
0&0
\\
$\gamma=0$&{\it poly3}& 
956&472&511&
0&939&0&
415&693
\\
&{\it poly4}& 
6&43&6&
5&2&15&
116&8
\\
&{\it poly5}& 
6&356&312&
967&37&982&
306&298
\\
&{\it poly6}& 
0&3&85&
20&1&3&
163&1
\\
\hline
&{\it poly1}& 
2&43&37&
2&35&49&
0&0
\\
&{\it poly2}& 
13&4&6&
173&44&46&
0&0
\\
$\gamma=1$&{\it poly3}& 
755&376&410&
756&606&514&
415&693
\\
&{\it poly4}& 
0&15&71&
0&0&1&
116&8
\\
&{\it poly5}& 
169&366&246&
10&166&213&
306&298
\\
&{\it poly6}& 
3&119&135&
1&35&49&
163&1
\\
\hline
\end{tabular}
\end{center}
\end{table}

In Table 3, we tabulate the number of times out of a 1000 repetitions that each polynomial model is selected based on bias reduction and information criteria. 
In multiplicative correction, {\it poly3} was selected by $C_a$, $C_\lambda$ and $C_{a\cap \lambda}$ most often. 
In additive correction of SPSE1, {\it poly3} was selected by $C_a$ most often. 
On the other hand, in SPSE3, $C_a$ and $C_{a\cap\lambda}$ selected {\it poly5}. 
Finally, AIC and TIC most often selected {\it poly3} and {\it poly5}. 
It appears that our criteria and the information criteria tend to choose the same model.

The ${\rm ISB}$, V and MISE of the estimators are shown in Table 4. 
In additive correction, {\it poly5} has the smallest ${\rm ISB}$. 
We note that $C_{a\cap\lambda}$ chooses {\it poly5} in SPSE3.
In both corrections, {\it poly3} has the smallest V and MISE in all models. 
On the whole, the SPSE displays better behavior than the SLLE although there are some exceptions. 
\\

\begin{table}
\begin{center}
\caption{Results of integrated squared bias, variance and mean integrated squared error for Example 2. All entries for ${\rm ISB}$,V and MISE are $10^3$ times their actual values.}
\scalebox{0.88}[0.9]{
\begin{tabular}{c|c|ccc|ccc|ccc}
\hline
\multicolumn{2}{c}{$n=25$}&\multicolumn{3}{|c|}{SPSE1}&\multicolumn{3}{c}{SPSE3}&\multicolumn{3}{|c}{SLLE}\\
\cline{3-8}
\hline
method&model&${\rm ISB}$&V&MISE&${\rm ISB}$&V&MISE&${\rm ISB}$&V&MISE\\
\hline
&{\it poly1}& 
1.213&232.429&233.643&
1.417&256.275&257.692&
1.991&246.245&248.236
\\
&{\it poly2}& 
0.846&226.256&227.103&
0.695&239.949&240.645&
 2.836&236.124& 238.960
\\
$\gamma=0$&{\it poly3}& 
0.776&225.508&226.285&
0.729&210.204&210.933&
1.157&243.466&244.623
\\
&{\it poly4}& 
1.322&251.572&252.894&
1.476&236.314&237.791&
 2.626&229.014&231.640
\\
&{\it poly5}& 
0.161&251.777&251.938&
0.122&238.596&238.717&
0.128&277.704&277.832
\\
&{\it poly6}& 
0.162&236.066&236.227&
0.134&233.793&233.927&
0.119&235.824&235.943
\\
\hline
&{\it poly1}& 
1.665&230.226&231.891&
1.746&253.074&254.820&
2.109&254.547&256.657
\\
&{\it poly2}& 
0.534&268.503&269.037&
0.321&225.818&226.138&
2.871&256.551& 259.421
\\
$\gamma=1$&{\it poly3}& 
0.323&213.758&214.081&
0.519&214.566&215.086&
1.545&237.094&238.638
\\
&{\it poly4}& 
0.924&233.528&234.452&
0.735&245.211&245.956&
2.858& 259.805&262.662
\\
&{\it poly5}& 
0.390&218.850&219.240&
0.624&221.162&221.786&
0.733& 243.170&243.903
\\
&{\it poly6}& 
0.356&241.451&241.807&
0.678&241.242&241.920&
0.895&240.767&241.662
\\
\hline
\hline
\multicolumn{2}{c}{$n=25$}&\multicolumn{3}{|c|}{NPSE1}&\multicolumn{3}{|c}{NPSE3}&\multicolumn{3}{|c}{NLLE}\\
\hline
\multicolumn{2}{c|}{Fully nonparametric}&ISB&V&MISE&ISB&V&MISE&ISB&V&MISE\\
\cline{3-11}
\multicolumn{2}{c|}{method}
&
1.213&232.429&233.643&
1.629&249.219&250.848&
1.991&246.245&248.236
\\
\hline
\end{tabular}
}
\end{center}
\end{table}

\noindent{\bf Example 3}
\quad The set-up of the true function and parametric models are the same as in Example 2, but the sample size is set to $n=75$. 
We utilize the error variance defined as $\sigma^2(x)=(x-0.5)^2+0.1$. 
However the parametric estimator is composed by the ordinary least squares method.

\begin{table}
\begin{center}
\caption{The results of parametric model selection in Example 3.}
\begin{tabular}{c|c|ccc|ccc|cc}
\hline
\multicolumn{2}{c|}{$n=75$}&\multicolumn{3}{c|}{SPSE1}&\multicolumn{3}{c|}{SPSE3}&\multicolumn{2}{c}{IC}\\
\cline{3-8}
\hline
method&model&$C_a$&$C_\lambda$&$C_{a\cap\lambda}$&$C_a$&$C_\lambda$&$C_{a\cap\lambda}$&AIC&TIC\\
\hline
&{\it poly1}& 
0&0&0&
0&0&0&
0&0
\\
&{\it poly2}& 
0&5&0&
0&65&0&
0&0
\\
$\gamma=0$&{\it poly3}& 
1000&47&8&
0&142&0&
457&0
\\
&{\it poly4}& 
0&2&172&
0&12&21&
94&2
\\
&{\it poly5}& 
0&604&630&
945&624&872&
296&950
\\
&{\it poly6}& 
0&277&113&
17&66&68&
153&48
\\
\hline
&{\it poly1}& 
8&2&8&
8&51&62&
0&0
\\
&{\it poly2}& 
64&222&168&
62&150&118&
0&0
\\
$\gamma=1$&{\it poly3}& 
894&17&86&
890&101&104&
457&0
\\
&{\it poly4}& 
0&72&104&
0&20&31&
94&2
\\
&{\it poly5}& 
0&363&398&
5&295&333&
296&950
\\
&{\it poly6}& 
0&182&85&
3&253&214&
153&48
\\
\hline
\end{tabular}
\end{center}
\end{table}

In Table 5, the results of the parametric model selection are shown. 
In additive correction of SPSE1, $C_{a\cap\lambda}$ indicates that the best model is {\it poly5} although $C_a$ selects {\it poly3} every time. 
In multiplicative correction, {\it poly3} is selected by $C_a$ many times while $C_\lambda$ and $C_{a\cap\lambda}$ select {\it poly5}.  
From the definition of $C_{a\cap\lambda}$, it is understood that {\it poly5} is selected in a fitness and smoothness context. 
On the other hand, AIC and TIC choose {\it poly3} and {\it poly5}, respectively. 
We note that the use of AIC might not be appropriate in this situation since the prepared model does not include the true $f$ and, hence, we place more confidence in TIC. 
On the other hand, when we select the parametric model only by the maximum of the log-likelihood, {\it poly5} was chosen 1000 times. 
Therefore, it seems that the bias correction in AIC is too strong in this situation. 

In Table 6, the ${\rm ISB}$, V and MISE of the SPSE are tabulated. 
In both corrections, the SPSE with {\it poly5} and {\it poly6} have overwhelmingly small ${\rm ISB}$s compared with those of {\it poly1-poly4}.  
As $C_a$ and $C_\lambda$ focus on bias reduction, it appears that $C_{a\cap\lambda}$ chooses {\it poly5} because it often has a small bias. 
On the other hand, {\it poly3} has good V and MISE, while {\it poly5} does not. 
For ${\rm ISB}$, V and MISE, the values of the SPSE is smaller than those of the SLLE, respectively. 
\\   

\begin{table}
\begin{center}
\caption{Results of integrated squared bias, variance and mean integrated squared bias of Example 3. All entries for ${\rm ISB}$, V and MISE are $10^3$ times their actual values.}
\scalebox{0.9}[0.9]{
\begin{tabular}{c|c|ccc|ccc|ccc}
\hline
\multicolumn{2}{c}{$n=75$}&\multicolumn{3}{|c|}{SPSE1}&\multicolumn{3}{c}{SPSE3}&\multicolumn{3}{|c}{SLLE}\\
\cline{3-8}
\hline
method&model&${\rm ISB}$&V&MISE&${\rm ISB}$&V&MISE&${\rm ISB}$&V&MISE\\
\hline
&{\it poly1}& 
0.061&1.330&1.390&
0.065&1.237&1.302&
0.645&6.529& 7.175
\\
&{\it poly2}& 
0.017&1.326&1.343&
0.007&1.231&1.238&
0.734&6.298&7.032
\\
$\gamma=0$&{\it poly3}& 
0.017&1.325&1.343&
0.007&1.230&1.237&
0.249&6.292&6.541
\\
&{\it poly4}& 
0.062&1.343&1.405&
0.066&1.251&1.317&
0.608&6.732&7.340
\\
&{\it poly5}& 
0.003&1.377&1.380&
0.002&1.285&1.287&
0.017&4.863&4.880
\\
&{\it poly6}& 
0.004&1.435&1.440&
0.002&1.350&1.354&
0.019&5.552&5.571
\\
\hline
&{\it poly1}& 
0.062&1.337&1.399&
0.068&1.246&1.314&
 1.084&6.167&7.251
\\
&{\it poly2}& 
0.024&1.328&1.352&
0.021&1.235&1.256&
0.997&6.186&7.183
\\
$\gamma=1$&{\it poly3}& 
0.030&1.325&1.342&
0.014&1.233&1.248&
0.314&6.279&6.593
\\
&{\it poly4}& 
0.072&1.348&1.419&
0.078&1.258&1.336&
0.420&6.476&6.896
\\
&{\it poly5}& 
0.003&1.380&1.383&
0.002&1.290&1.292&
0.023&4.925&4.949
\\
&{\it poly6}& 
0.003&1.438&1.441&
0.002&1.353& 1.355&
0.025 & 5.528&5.553
\\
\hline
\hline
\multicolumn{2}{c}{$n=75$}&\multicolumn{3}{|c|}{NPSE1}&\multicolumn{3}{|c}{NPSE3}&\multicolumn{3}{|c}{NLLE}\\
\hline
\multicolumn{2}{c|}{Fully nonparametric}&ISB&V&MISE&ISB&V&MISE&ISB&V&MISE\\
\cline{3-11}
\multicolumn{2}{c|}{method}
&
0.061&1.330&1.390&
0.065&1.237&1.302&
0.645&6.529& 7.175
\\
\hline
\end{tabular}
}
\end{center}
\end{table}

\noindent{\bf Example 4}
\quad The true model is $f(x)=4+e^{-x}\{\sin(7\pi x)+2\cos(3\pi x)\}$ and the error variance is $\sigma^2(x)=0.5$. 
The parametric model is 
\begin{eqnarray*}
f(x|\vec{\beta})=
\left\{
\begin{array}{ll}
\beta_0+e^{-x}\{\beta_1+\beta_2\sin(7\pi x)+\beta_3\cos(3\pi x)\},&f_1={\it sincos},\\
\beta_0+e^{-x}\{\beta_1+\beta_2\sin(7\pi x)\},&f_2={\it sin},\\
\beta_0+e^{-x}\{\beta_1+\beta_2\cos(3\pi x)\},&f_3={\it cos},\\
\beta_0+e^{-x}\{\beta_1+\beta_2x\},&f_4={\it poly1},\\
\beta_0+e^{-x}\{\beta_1+\beta_2x+\cdots+\beta_5x^4\},&f_5={\it poly4},\\
\beta_0+e^{-x}\{\beta_1+\beta_2x+\cdots+\beta_9x^8\},&f_6={\it poly8}
\end{array}
\right.
\end{eqnarray*}
The function {\it sincos} corresponds to the true function. 

\begin{table}
\begin{center}
\caption{The results of parametric model selection in Example 4.}
\begin{tabular}{c|c|ccc|ccc|cc}
\hline
\multicolumn{2}{c}{$n=50$}&\multicolumn{3}{|c|}{SPSE1}&\multicolumn{3}{c|}{SPSE3}&\multicolumn{2}{c}{IC}\\
\cline{3-8}
\hline
method&model&$C_a$&$C_\lambda$&$C_{a\cap\lambda}$&$C_a$&$C_\lambda$&$C_{a\cap\lambda}$&AIC&TIC\\
\hline
&{\it sin}& 
997&998&992&
987&996&972&
1&0
\\
$\gamma=0$&{\it cos}& 
3&2&0&
7&4&17&
602&11
\\
&{\it poly1}& 
0&0&0&
1&0&0&
0&0
\\
&{\it poly4}& 
0&0&2&
0&0&0&
397&902
\\
&{\it poly8}& 
0&0&3&
0&0&0&
0&87
\\
\hline
&{\it sin}& 
887&823&686&
887&821&791&
1&0
\\
$\gamma=1$&{\it cos}& 
77&17&1&
77&114&43&
602&11
\\
&{\it poly1}& 
0&11&37&
0&0&0&
0&0
\\
&{\it poly4}& 
0&56&109&
0&23&93&
397&902
\\
&{\it poly8}& 
0&47&88&
0&14&28&
0&87
\\
\hline
\end{tabular}
\end{center}
\end{table}

In Table 7, the results of the parametric model selection are tabulated. 
The {\it sincos}, corresponding to the true $f$, was not included in the model selection since it should be chosen frequently.
In both corrections, $\gamma=0, 1$, {\it sin} was chosen by $C_a$, $C_\lambda$ and $C_{a\cap\lambda}$ most often. 
On the other hand, TIC selected {\it poly4}, and AIC selected {\it cos} and {\it poly4} quit often. 

In Table 8, the ${\rm ISB}$, V and MISE of the estimators are shown.   
In both corrections, $\gamma=0,1$, the behavior of the SPSE with {\it sin} is superior than that of the SPSE with any other model except {\it sincos}. 
We observe that the SPSE with the initial parametric model selected by $C_{a\cap\lambda}$ shows better behavior than that with the model selected by information criteria.

Furthermore it can be seen that ${\rm ISB}$, V and MISE of the SLLE with {\it sincos} are significantly smaller than those of the SPSE with any parametric model. 
On the other hand, if we use incorrect models (other than {\it sincos}) in the SLLE, then the ${\rm ISB}$, V and MISE of the SLLE are larger than those of the SPSE. 
\\

\begin{table}
\begin{center}
\caption{Results of integrated squared bias, variance and mean integrated squared error for Example 4. All entries for ${\rm ISB}$,V and MISE are $10^3$ times their actual values.}
\scalebox{0.88}[0.9]{
\begin{tabular}{c|c|ccc|ccc|ccc}
\hline
\multicolumn{2}{c}{$n=50$}&\multicolumn{3}{|c|}{SPSE1}&\multicolumn{3}{c}{SPSE3}&\multicolumn{3}{|c}{SLLE}\\
\cline{3-8}
\hline
method&model&ISB&V&MISE&ISB&V&MISE&ISB&V&MISE\\
\hline
&{\it sincos}& 
0.051&87.361&87.412&
0.041&81.564&81.605&
0.025&64.752&64.777
\\
&{\it sin}& 
2.689&86.891&89.580&
3.270&81.053&84.323&
15.416 & 85.149&100.566
\\
$\gamma=0$&{\it cos}& 
17.206&87.095&104.302&
13.195&86.217& 99.411&
21.039&92.615&113.654
\\
&{\it poly1}& 
19.095&89.314&108.409&
13.950&88.674&102.624&
25.920&104.183&130.103
\\
&{\it poly4}& 
15.990&91.930&107.920&
11.733&90.234&101.967&
25.716&106.923&132.639
\\
&{\it poly8}& 
16.492&94.013&110.505&
11.896&92.078&103.975&
22.992& 108.436&131.428
\\
\hline
&{\it sincos}& 
0.051&88.492&88.543&
0.040&82.978&83.018&
0.025&63.735& 63.761
\\
&{\it sin}& 
4.968&87.858&92.825&
6.245&82.485&88.730&
18.049&83.225&101.274
\\
$\gamma=1$&{\it cos}& 
17.269&89.904&107.174&
12.525&89.165&101.690&
 20.751 & 92.491&113.242
\\
&{\it poly1}& 
18.981&90.991& 109.972&
13.360&90.053&103.413&
28.430&94.194&122.624
\\
&{\it poly4}& 
15.451& 94.073&109.524&
11.155&92.079&103.233&
24.959&106.714&131.673
\\
&{\it poly8}& 
15.534&95.991&111.525&
10.936&93.630& 104.566&
26.838& 106.554&133.392
\\
\hline
\hline
\multicolumn{2}{c}{$n=50$}&\multicolumn{3}{|c|}{NPSE1}&\multicolumn{3}{|c}{NPSE3}&\multicolumn{3}{|c}{NLLE}\\
\hline
\multicolumn{2}{c|}{Fully nonparametric}&ISB&V&MISE&ISB&V&MISE&ISB&V&MISE\\
\cline{3-11}
\multicolumn{2}{c|}{method}
&
18.884&88.770&107.653&
13.878&88.201&102.079&
26.859&93.344&120.204
\\
\hline
\end{tabular}
}
\end{center}
\end{table}

\noindent {\bf Remark 7}
\quad In all examples, we also compared the behavior of the SPSE and the SLLE under the conditions that $K_n$ is equal to the ceiling of $h_n^{-1}$ and that $\lambda_n=n^p/n^{2p+1}$. 
From these results, we have confirmed that
the ${\rm ISB}$ of the SPSE is smaller than that of the SLLE for each parametric model. 
In contrast, the V and MISE of the SPSE are larger than those of the SLLE. 
Thus, it seems that the SPSE produces overfitting.


\section{Discussion}

We have discussed the SPSE using a parametric model. 
We see that the SPSE has better behavior than the NPSE, provided we can choose a good $f(x|\vec{\beta})$ in the first parametric step.
A similar conclusion can be drawn for the semiparametric regression spline estimator by letting $\lambda_n=0$. 

In the field of kernel smoothing, Fan et al. (2009) noted that the semiparametric local polynomial estimator can also be constructed in the additive model (Hastie and Tibshirani (1990)).
The reason for this is the asymptotic result of nonparametric kernel regression in the additive model, which has previously been developed by Ruppert and Opsomer (1997) and Opsomer (2000). 
On the other hand, it appears that the asymptotic results for the penalized spline estimator have still not been sufficiently investigated in comparison to kernel smoothing. 
While it is beyond the scope of this paper, this semiparametric approach with a penalized spline can be also extended to the generalized linear model.
In this sense, there are still many topics that should be examined in theoretical studies of the penalized spline method. 


\section*{Appendix}


For a matrix $A_n=(a_{ij,n})_{ij}$, if $\displaystyle\max_{i,j}\{n^{\alpha}|a_{ij,n}|\}=O_{P}(1)(o_{P}(1))$, then it is written as $a_n=O_{P}(n^{-\alpha}\vec{1}\vec{1}^\prime)(o_{P}(n^{-\alpha}\vec{1}\vec{1}^\prime))$. 
When $A_n$ is vector, define $A_n=O_P(n^{-\alpha}\vec{1}) (o_P(n^{-\alpha}\vec{1}))$ like a matrix case.
This notation will be used for matrices with fixed sizes and sizes depending on $n$.
For the proofs of Proposition 1, Theorems 1-2 and Corollary 1, we define $\Lambda_{n}=n^{-1}\Lambda$. 
We need additional lemmas as follows.  

\vspace{5mm}

\begin{lemma}\label{Lam}
Let $A=(a_{ij})_{ij}$ be $(K_n+p)$ matrix. 
Assume that $K_n\rightarrow \infty$ as $n\rightarrow \infty$, $A=O_P(K_n^\alpha\vec{1}\vec{1}^\prime)$. 
Then $A\Lambda_n^{-1}=O(K_n^{1+\alpha}\vec{1}\vec{1}^\prime)$
\end{lemma}

\vspace{5mm}

\begin{lemma}\label{spinte}
Let $g : \mathbb{R}\rightarrow \mathbb{R}$ be any function with $\displaystyle\sup_{x\in\mathbb{R}}\{g(x)\}<\infty$. 
Then, $\int_0^1 B_i(u)g(u) du =O(K_n^{-1})$ and 
$
\int_0^1 B_i(u)B_j(u)g(u) du =O(K_n^{-1}).
$
\end{lemma}

Lemmas \ref{Lam} and \ref{spinte} are shown by fundamental properties of $B$-spline(see, Claeskens et al. (2009) and Zhou et al. (1998)).

\vspace{5mm}

\begin{proof}[Proof of Proposition \ref{fix}]

First we calculate the asymptotic expectation of $\hat{r}_\gamma(x,\vec{\beta}_0)$: 
\begin{eqnarray*}
E[\hat{r}_\gamma(x,\vec{\beta}_0)|\vec{X}_n]=f(x|\vec{\beta}_0)^\gamma\vec{B}(x)^\prime \Lambda^{-1}Z^\prime E[\vec{r}_\gamma|\vec{X}_n],
\end{eqnarray*}
where 
$$
E[\vec{r}_\gamma|\vec{X}_n]=\left(\frac{f(x_1)-f(x_1|\vec{\beta}_0)}{f(x_1|\vec{\beta}_0)^\gamma}\ \cdots\ \frac{f(x_n)-f(x_n|\vec{\beta}_0)}{f(x_n|\vec{\beta}_0)^\gamma}\right)^\prime
$$
By using Theorem 2 (a) of Claeskens et al. (2009), if  $\{f(x)-f(x|\vec{\beta}_0)\}/f(x|\vec{\beta}_0)^\gamma$ is regarded as regression function, we have
$$
E[\hat{r}_\gamma(x,\vec{\beta}_0)|\vec{X}_n]=\frac{f(x)-f(x|\vec{\beta}_0)}{f(x|\vec{\beta}_0)^\gamma}+b_{a1}(x|\vec{\beta}_0,\gamma)+b_{\lambda 1}(x|\vec{\beta}_0,\gamma)+o_P(K_n^{-(p+1)})+o_P(\lambda_nK_nn^{-1}),
$$
where 
$b_{\lambda 1}(x|\vec{\beta}_0,\gamma)=-(\lambda_n/n)\vec{B}(x)^\prime G(q)^{-1}Q_m \vec{b}^*(\vec{\beta}_0,\gamma)$.
Therefore, the expectation of $\hat{f}_0(x,\gamma)$ can be written as 
\begin{eqnarray*}
E[\hat{f}_0(x,\gamma)|\vec{X}_n]
&=&f(x|\vec{\beta}_0)+f(x|\vec{\beta}_0)^\gamma E[\hat{r}_\gamma(x,\vec{\beta}_0)|\vec{X}_n]\\
&=&f(x)+f(x|\vec{\beta}_0)^\gamma\{b_{a1}(x|\vec{\beta}_0,\gamma)+b_{\lambda 1}(x|\vec{\beta}_0,\gamma)\}\\
&&+o_P(K_n^{-(p+1)})+o_P(\lambda_nK_nn^{-1})\\
&=&f(x)+b_{a}(x|\vec{\beta},\gamma)+b_{\lambda }(x|\vec{\beta},\gamma)+o_P(K_n^{-(p+1)})+o_P(\lambda_nK_nn^{-1}).
\end{eqnarray*}
Next we show the asymptotic variance of $\hat{f}_0(x,\gamma)$. 
It is easy to see that  
\begin{eqnarray*}
V[\hat{f}_0(x,\gamma)|\vec{X}_n]&=&f(x|\vec{\beta})^{2\gamma} \vec{B}(x)^\prime \Lambda^{-1}Z^\prime V[\vec{r}_\gamma|\vec{X}_n]Z\Lambda^{-1}\vec{B}(x)\\
&=&\frac{f(x|\vec{\beta})^{2\gamma} }{n^2}\vec{B}(x)^\prime \Lambda_n^{-1}Z^\prime \left(\diag\left[\frac{\sigma^2(x_1)}{f(x_1|\vec{\beta})^{2\gamma}},\cdots,\frac{\sigma^2(x_n)}{f(x_n|\vec{\beta})^{2\gamma}}\right]\right)Z\Lambda_n^{-1}\vec{B}(x).
\end{eqnarray*}
The $(i,j)$-component of $n^{-1}Z^\prime V[\vec{r}_\gamma|\vec{X}_n]Z$ can be calculated as 
\begin{eqnarray*}
&&\left(\frac{1}{n}Z^\prime \left(\diag\left[\frac{\sigma^2(x_1)}{f(x_1|\vec{\beta})^2},\cdots,\frac{\sigma^2(x_n)}{f(x_n|\vec{\beta})^2}\right]\right)Z\right)_{ij}\\
&&=\frac{1}{n}\sum_{k=1}^n B_{-p+i}^{[p]}(x_k)B_{-p+j}^{[p]}(x_k)\frac{\sigma^2(x_k)}{f(x_k|\vec{\beta})^2}\\
&&=\int_0^1 B_{-p+i}^{[p]}(u)B_{-p+j}^{[p]}(u)\frac{\sigma^2(u)q(u)}{f(u|\vec{\beta})^2}du(1+o_P(1)).
\end{eqnarray*}
Hence, we obtain
\begin{eqnarray*}
V[\hat{f}_0(x,\gamma)|\vec{X}_n]=\frac{f(x|\vec{\beta})^{2\gamma}}{n}\vec{B}(x)^\prime G(q)^{-1}G(\sigma,\beta,\gamma,q) G(q)^{-1}\vec{B}(x)+o_P(K_nn^{-1}).
\end{eqnarray*}
\end{proof}

Before proof of Theorem \ref{est}, we define some symbols. 
For any function $g(\cdot|\vec{\beta})$ which is smooth for $\vec{\beta}$, 
$$
g^{(1)}(\cdot|\vec{\beta}_0)=\frac{\partial g(\cdot|\vec{\beta})}{\partial \vec{\beta}}\Bigl{|}_{\vec{\beta}=\vec{\beta}_0},\ \ 
g^{(2)}(\cdot|\vec{\beta}_0)=\frac{\partial^2 g(\cdot|\vec{\beta})}{\partial \vec{\beta}\partial \vec{\beta}^\prime }\Bigl|_{\vec{\beta}=\vec{\beta}_0}.
$$
We use Taylor expansion of $g(\cdot|\hat{\vec{\beta}})$ around $\vec{\beta}_0$, giving
\begin{eqnarray}
g(\cdot|\hat{\vec{\beta}})
=
g(\cdot|\vec{\beta}_0)+
g^{(1)}(\cdot|\vec{\beta}_0)^\prime(\hat{\vec{\beta}}-\vec{\beta}_0)+\frac{1}{2}(\hat{\vec{\beta}}-\vec{\beta}_0)^\prime g^{(2)}(\cdot|\vec{\beta}_0) (\hat{\vec{\beta}}-\vec{\beta}_0)+o_P(n^{-1}).\label{Tay}
\end{eqnarray}

\begin{proof}[Proof of Theorem \ref{est}]

We first note from (\ref{semiest}) that the SPSE is expressed as
$$
\hat{f}(x,\gamma)=f(x|\hat{\vec{\beta}})+\vec{B}(x)^\prime\Lambda^{-1}Z^\prime \vec{r}_\gamma(\hat{\vec{\beta}}),
$$
where 
\begin{eqnarray*}
\vec{r}_\gamma(\hat{\vec{\beta}})
&=&(r_\gamma(y_1|\hat{\vec{\beta}})\ \cdots\ r_\gamma(y_n|\hat{\vec{\beta}}))^\prime
\end{eqnarray*}
and $r_\gamma(y_i|\hat{\vec{\beta}})=f(x|\hat{\vec{\beta}})^\gamma\{y_i-f(x_i|\hat{\vec{\beta}})\}/f(x_i|\hat{\vec{\beta}})^\gamma$.

Taylor expansion yields that
\begin{eqnarray}
\hat{f}(x,\gamma)=\hat{f}_{0}(x,\gamma)+\hat{f}^{(1)}(x,\gamma)^{\prime}(\hat{\vec{\beta}}-\vec{\beta}_0)+\frac{1}{2}(\hat{\vec{\beta}}-\vec{\beta}_0)^{\prime} \hat{f}^{(2)}(x,\gamma)(\hat{\vec{\beta}}-\vec{\beta}_0)+o_P(n^{-1}), \label{fhatex}
\end{eqnarray}
where
$$
\hat{f}^{(1)}(x,\gamma)=f^{(1)}(x|\vec{\beta}_{0})+\sum_{j=1}^{n}\left\{\vec{B}(x_{j})^{\prime}\Lambda^{-1}\vec{B}(x)\right\}r_{\gamma}^{(1)}(y_{j}|\vec{\beta}_{0})
$$
and
$$
\hat{f}^{(2)}(x,\gamma)=f^{(2)}(x|\vec{\beta}_{0})+\sum_{j=1}^{n}\left\{\vec{B}(x_{j})^{\prime}\Lambda^{-1}\vec{B}(x)\right\}r_{\gamma}^{(2)}(y_{j}|\vec{\beta}_{0}).
$$
First we derive the asymptotic expectation of $\hat{f}(x,\gamma)$.
The term $E[\hat{f}_{0}(x,\gamma)|\vec{X}_{n}]$ has already been derived in Proposition 1.
Direct calculations with repeated use of (\ref{ex}) and Lemmas 1 and 2 yield that 
\begin{eqnarray*}
\frac{1}{n}\sum_{\alpha=1}^n E\left[f^{(1)}(x|\vec{\beta}_{0})^\prime \left.\left\{I(x_\alpha,Y_\alpha)+\frac{d}{n}+\delta_n\right\}\right|\vec{X}_n\right]&=&\frac{1}{n}E[f^{(1)}(x|\vec{\beta}_{0})^\prime d|\vec{X}_n]+O(n^{-2})\\
&=&O(n^{-1})
\end{eqnarray*}
and 
\begin{eqnarray*}
&&\frac{1}{n}\sum_{\alpha=1}^n \sum_{j=1}^{n}\left\{\vec{B}(x_{j})^{\prime}\Lambda^{-1}\vec{B}(x)\right\}E\left[r^{(1)}_\gamma(Y_j|\vec{\beta}_{0})^\prime \left.\left\{I(x_\alpha,Y_\alpha)+\frac{d}{n}+\delta_n\right\}\right|\vec{X}_n\right]\\
&&=\frac{1}{n}\sum_{j=1}^{n}\left\{\vec{B}(x_{j})^{\prime}\Lambda^{-1}\vec{B}(x)\right\}E\left[r^{(1)}_\gamma(Y_j|\vec{\beta}_{0})^\prime \left.\left\{I(x_j,Y_j)+\frac{d}{n}\right\}\right|\vec{X}_n\right]+O_P(n^{-2})\\
&&=O_P(n^{-1}).
\end{eqnarray*}
Hence we obtain 
\begin{eqnarray}
E[\hat{f}^{(1)}(x,\gamma)^{\prime}(\hat{\vec{\beta}}-\vec{\beta}_0) | \vec{X}_{n}]=O_P(n^{-1}). \label{ex1}
\end{eqnarray}
Analogously, 
\begin{eqnarray}
E[(\hat{\vec{\beta}}-\vec{\beta}_0)^{\prime} \hat{f}^{(2)}(x,\gamma)(\hat{\vec{\beta}}-\vec{\beta}_0) | \vec{X}_{n}]=O_P(n^{-1}) \label{ex2}
\end{eqnarray}
can be also shown. 
(\ref{ex1}) and (\ref{ex2}) are smaller order than the bias terms of $\hat{f}_{0}(x,\gamma)$.
Therefore the bias of $\hat{f}(x,\gamma)$ is essentially dominated by the bias of $\hat{f}_{0}(x,\gamma)$.

Next we turn to the variance of $\hat{f}(x,\gamma)$.
It follows from direct evaluation using (\ref{ex}) that
$$
V[\hat{f}^{(1)}(x,\gamma)^{\prime}(\hat{\vec{\beta}}-\vec{\beta}_0) | \vec{X}_{n}]=O_P(n^{-1}).
$$
And simple but tedious calculations finally yield
$$
V[(\hat{\vec{\beta}}-\vec{\beta}_0)^{\prime} \hat{f}^{(2)}(x,\gamma)(\hat{\vec{\beta}}-\vec{\beta}_0) | \vec{X}_{n}]=O_P(n^{-2}).
$$
All terms of relating to covariance appeared from the right hand side of (\ref{fhatex}) can be shown to be negligible order by Cauchy-Schwarz inequality. 
Hence the variance of $\hat{f}(x,\gamma)$ is dominated by that of $\hat{f}_{0}(x,\gamma)$.
\end{proof}

\begin{proof}[Proof of Theorem \ref{norm}]
Let $\hat{r}(x,\gamma)=\vec{B}(x)^\prime\Lambda^{-1}Z^\prime \vec{r}_\gamma(\hat{\vec{\beta}})$. 
Then the semiparametric estimator can be written as 
$\hat{f}(x,\gamma)=f(x|\hat{\vec{\beta}})+\hat{r}(x,\gamma)$. 
We now prove 
\begin{eqnarray}
\frac{\hat{f}(x,\gamma)-E[\hat{f}(x,\gamma)|\vec{X}_n]}{\sqrt{V[\hat{f}(x,\gamma)|\vec{X}_n]}}\xrightarrow {D}N(0,1) \label{fnorm}
\end{eqnarray}
by using Lyapunov theorem. 
First, from $\sqrt{n}(f(x|\hat{\vec{\beta}})-E[f(x|\hat{\vec{\beta}})|\vec{X}_n])=O_P(1)$ and $V[\hat{f}(x,\gamma)|\vec{X}_n]=O(K_nn^{-1})$, 
we have 
$$
\frac{f(x|\hat{\vec{\beta}})-E[f(x|\hat{\vec{\beta}})|\vec{X}_n]}{\sqrt{V[\hat{f}(x,\gamma)|\vec{X}_n]}} \xrightarrow {P} 0.
$$
Therefore, (\ref{fnorm}) can be obtained, provided that
\begin{eqnarray}
\frac{\hat{r}(x,\gamma)-E[\hat{r}(x,\gamma)|\vec{X}_n]}{\sqrt{V[\hat{r}(x,\gamma)|\vec{X}_n]}}
\xrightarrow {D}N(0,1) \label{rnorm}
\end{eqnarray}
because $V[\hat{f}(x,\gamma)|\vec{X}_n]/V[\hat{r}(x,\gamma)|\vec{X}_n]\rightarrow 1 (n\rightarrow \infty)$. 
Furthermore, from the proof of Theorem \ref{est}, we obtain 
\begin{eqnarray*}
\frac{\hat{r}(x,\gamma)-\hat{r}_0(x,\gamma)}{\sqrt{V[\hat{r}(x,\gamma)|\vec{X}_n]}}
\xrightarrow {P} 0,\ \ {\rm as}\ \ n\rightarrow \infty
\end{eqnarray*}
and $V[\hat{r}(x,\gamma)|\vec{X}_n]/V[\hat{r}_0(x,\gamma)|\vec{X}_n]\rightarrow 1 (n\rightarrow \infty)$,
where 
\begin{eqnarray*}
\hat{r}_0(x,\gamma)=\vec{B}(x)^\prime\Lambda^{-1}Z^\prime \vec{r}_\gamma(\vec{\beta}_0)=f(x|\vec{\beta}_0)^\gamma\sum_{i=1}^{n} \{\vec{B}(x_{i})^{\prime}\Lambda^{-1}\vec{B}(x)\}\frac{ \{y_i-f(x_i|\vec{\beta}_0)\} }{f(x_i|\vec{\beta}_0)^\gamma}.
\end{eqnarray*}
From now on, we try to show 
\begin{eqnarray}
\frac{\hat{r}_0(x,\gamma)-E[\hat{r}_0(x,\gamma)|\vec{X}_n]}{\sqrt{V[\hat{r}_0(x,\gamma)|\vec{X}_n]}}\xrightarrow {D} N(0,1) \label{r0norm}
\end{eqnarray}
by applying the Lyapunov theorem.
First we see that
$$
\hat{r}_0(x,\gamma)-E[\hat{r}_0(x,\gamma)|\vec{X}_n]=f(x|\vec{\beta}_0)^\gamma\sum_{i=1}^{n} \{\vec{B}(x_{i})^{\prime}\Lambda^{-1}\vec{B}(x)\}\frac{\varepsilon_{i}}{f(x_i|\vec{\beta}_0)^\gamma}.
$$
And it is easily confirmed that
$$
f(x|\vec{\beta}_0)^\gamma\vec{B}(x)^\prime\Lambda^{-1}\vec{B}(x_i)=O_P(K_{n}n^{-1}).
$$
By above evaluations and the moment condition for $\varepsilon_{i}$, we have
\begin{eqnarray*}
&&E\left[  \left|f(x|\vec{\beta}_0)^\gamma\{\vec{B}(x_{i})^{\prime}\Lambda^{-1}\vec{B}(x)\}\frac{\varepsilon_{i}}{f(x_i|\vec{\beta}_0)^\gamma}\right|^{2+\delta}    \left| \vec{X}_{n}\frac{}{}\right.\right]\\
&&=
\frac{E[|f(x|\vec{\beta}_0)^\gamma\vec{B}(x)^\prime\Lambda^{-1}\vec{B}(x_i)\varepsilon_i|^{2+\delta}|\vec{X}_n]}{|f(x_i|\vec{\beta}_0)|^{\gamma(2+\delta)}} \\
&&=O_P\left(\frac{K_n^{2+\delta}}{n^{2+\delta}}\right).
\end{eqnarray*}
On the other hand, since
$B_n^{2}=V[\hat{r}_0(x,\gamma)|\vec{X}_n]=O_P(K_nn^{-1})$, 
we have 
$$
B_n^{2+\delta}=O_P\left(\left(\frac{K_n}{n}\right)^{(2+\delta)/2}\right).
$$
Then it follows that
\begin{eqnarray}
\ &\ &\frac{1}{B_n^{2+\delta}}\sum_{i=1}^n E\left[ \left| f(x|\vec{\beta}_0)^\gamma\{\vec{B}(x_{i})^{\prime}\Lambda^{-1}\vec{B}(x)\}\frac{\varepsilon_{i}}{f(x_i|\vec{\beta}_0)^\gamma} \right|^{2+\delta} \left|X_{i}\frac{}{}\right.  \right]  \nonumber \\
\ &\ &=O_P\left(n\left(\frac{K_n}{n}\right)^{2+\delta}\right) O_P\left(\left(\frac{K_n}{n}\right)^{-(2+\delta)/2}\right)\nonumber\\
\ &\ &=O_P\left(n\left(\frac{K_n}{n}\right)^{\frac{2+\delta}{2}}\right), \nonumber
\label{conde}
\end{eqnarray}
which tends to 0 in probability by $K_n=o(n^{1/2})$ and $\delta\geq 2$.
This assures the Lyapunov condition, so that (\ref{r0norm}) holds. 
Note that $b_a(x|\vec{\beta}_0,\gamma)=O(K_n^{-(p+1)})$, $b_\lambda(x|\vec{\beta}_0,\gamma)=O(\lambda_nK_nn^{-1})$ and  $V[\hat{f}(x,\gamma)|\vec{X}_n]=O(K_nn^{-1})$.
It results from these evaluations and the assumptions for the order of $K_{n}$ and $\lambda_{n}$ that
$$
\frac{ E[\hat{f}(x,\gamma)|\vec{X}_{n}]-f(x)-b_a(x|\vec{\beta}_0,\gamma)-b_\lambda(x|\vec{\beta}_0,\gamma) }{ \sqrt{V[\hat{f}(x,\gamma)|\vec{X}_n]} } \rightarrow 0,
$$
which completes the proof.
\end{proof}

\begin{proof}[Proof of Corollary \ref{pol}]

First, $f_q(x|\vec{\beta}_q)$ can be expressed as the linear combination of the $p$th $B$-spline basis. 
From the fundamental property of $B$-spline basis (see, p.95 of de Boor (2001)), actually, each $x^j$ can be written as 
$$
x^{p-j}=\sum_{k=-p+1}^{K_n}\frac{(-1)^j(p-j)!}{p!}\phi^{(j)}_{k,p}(0)B_k^{[p]}(x),\ \ j=p-q,\cdots,p,
$$ 
where $\phi_{k,p}(z)=(\kappa_{k}-z)\cdots(\kappa_{k+p-1}-z)$ and we have
\begin{eqnarray}
f_q(x|\vec{\beta}_q)&=&\beta_0+\beta_1x+\cdots+\beta_qx^q \nonumber\\
&=&\sum_{j=p-q}^p \beta_{p-j}x^{p-j}\nonumber\\
&=&\sum_{k=-p+1}^{K_n}\left\{\sum_{j=p-q}^p\beta_{p-j}\frac{(-1)^j(p-j)!}{p!}\phi^{(j)}_{k,p}(0)\right\}B_k^{[p]}(x). \label{polysp}
\end{eqnarray}
Note that (\ref{polysp}) consist for any $\vec{\beta}\in B \subseteq \mathbb{R}^{q+1}$.
The semiparametric penalized spline estimator is obtained by 
$
\hat{f}(x,0)=f_q(x|\hat{\vec{\beta}}_q)+\hat{r}_0(x,\hat{\vec{\beta}}_q).
$
Let $\hat{\vec{c}}=(\hat{c}_{-p+1}\ \cdots\ \hat{c}_{K_n})^\prime$ be the $(K_n+p)$ vector defined as  
$$
\hat{c}_k=\sum_{j=p-q}^p\hat{\beta}_{p-j}\frac{(-1)^j(p-j)!}{p!}\phi^{(j)}_{k,p}(0),\ \ k=-p+1,\cdots,K_n
$$ 
Then, we have $f_q(x|\hat{\vec{\beta}}_q)=\vec{B}(x)^\prime \hat{\vec{c}}$ and 
$$
\hat{r}_0(x,\hat{\vec{\beta}}_q)=\vec{B}(x)^\prime \hat{\vec{b}} =\vec{B}(x)^\prime (Z^\prime Z+\lambda_n Q_m)^{-1}Z^\prime (\vec{y}-Z\hat{\vec{c}}).
$$
Therefore, we have 
\begin{eqnarray}
\hat{f}(x,0)=f_q(x|\hat{\vec{\beta}}_q)+\hat{r}_0(x,\hat{\vec{\beta}}_q) =\vec{B}(x)^\prime \hat{\vec{c}}+\vec{B}(x)^\prime (Z^\prime Z+\lambda_n Q_m)^{-1}Z^\prime (\vec{y}-Z\hat{\vec{c}}). \label{fhatpoly}
\end{eqnarray}
When $\lambda_n=0$, meaning that $\hat{r}_0(x,\hat{\vec{\beta}}_q)$ is regression spline, (\ref{fhatpoly}) can be written as 
$$
\hat{f}(x,0)=\vec{B}(x)^\prime \hat{\vec{c}}+\vec{B}(x)^\prime (Z^\prime Z)^{-1}Z^\prime (\vec{y}-Z\hat{\vec{c}})=\vec{B}(x)^\prime (Z^\prime Z)^{-1}Z^\prime \vec{y}
$$
for all $p\geq 1$. 
So the semiparametric estimator and nonparametric estimator have the same form.
If $\lambda_n>0$, on the other hand,  
\begin{eqnarray*}
\hat{f}(x,0)=\vec{B}(x)^\prime \hat{\vec{c}}-\vec{B}(x)^\prime (Z^\prime Z+\lambda_n Q_m)^{-1}Z^\prime Z\hat{\vec{c}}
=\lambda_n\vec{B}(x)^\prime (Z^\prime Z+\lambda_n Q_m)^{-1} Q_m\hat{\vec{c}}
\end{eqnarray*}
does not become 0 unless $Q_m\hat{\vec{c}}=\vec{0}$. 
However as far as we use $(p,m)=(1,2)$ and equidistant knots, we obtain $Q_m\hat{\vec{c}}=\vec{0}$. 
The square matrix $Q_2$ of order $(K_n+p)$ has the form $Q_2=D_2^\prime D_2$, where $(K_n+p-2)\times (K_n+p)$ matrix $D_2$ is 
\begin{eqnarray*}
D_2=(d_{ij})_{ij}=
\left[
\begin{array}{cccccc}
1&-2&1&0&\cdots &0\\
0&1&-2&1&\ddots &\vdots\\
\vdots &\ddots &\ddots &\ddots &\ddots&\vdots\\
0&\cdots &0&1&-2&1
\end{array}
\right].
\end{eqnarray*}
We way only prove $D_2\hat{\vec{c}}=\vec{0}$. 
Because the $k$th component of $\hat{\vec{c}}$ is 
$$
\sum_{j=0}^p\hat{\beta}_{p-j}\frac{(-1)^j(p-j)!}{p!}\phi^{(j)}_{k,p}(0),
$$
we show that for $j=0,1$ and $p=1$,  
$$
\sum_{k=-p+1}^{K_n} d_{ik} \phi^{(j)}_{k,1}(0)=0,\ \ i=1,\cdots,K_n+p.
$$
By the definition of $d_{ik}$ and $\phi^{(j)}_{k,1}(z)=(\kappa_k-z)^{(j)}$, we have for $j=0$, 
\begin{eqnarray*}
\sum_{k=-p+1}^{K_n} d_{ik} \phi^{(0)}_{k,1}(0)
&=&d_{i,i} \kappa_i+d_{i,i+1} \kappa_{i+1}+d_{i,i+2} \kappa_{i+2}\\
&=&0.
\end{eqnarray*}
For $j=1$, we obtain $\sum_{k=-p+1}^{K_n} d_{ik} \phi^{(1)}_{k,1}(0)=0$. 
Therefore, $D_2\hat{\vec{c}}=\vec{0}$ was proven.
\end{proof}

\noindent {\bf References}
\vskip 1mm

\noindent Claeskens,G., Krivobokova,T. and Opsomer,J.D. (2009). Asymptotic properties of penalized spline estimators.\ $Biometrika.$ $\mathbf{96}$, 529-544.

\noindent  de Boor,C. (2001). $A\ Practical\ Guide\ to\ Splines$. Springer-Verlag.

\noindent Eilers,P.H.C. and Marx,B.D. (1996). Flexible smoothing with $B$-splines and penalties(with Discussion). $Statist.Sci$. {\bf 11}, 89-121.

\noindent  Fan,J., Wu,Y. and Feng,Y. (2009). Local quasi-likelihood with a parametric guide. {\it Ann. Statist.} {\bf 37} 4153-4183.

\noindent  Glad,I.K. (1998). Parametrically guided non-parametric regression. {\it Scand.J.Statist.} {\bf 25} 649-668.

\noindent  Hastie,T. and Tibshirani,R.(1990). {\it Generalized Additive Models}. London Chapman \& Hall.

\noindent  Hjort,N.L. and Glad,I.K. (1995). Nonparametric density estimation with a parametric start. {\it Ann. Statist.} {\bf 23} 882-904.

\noindent  Konishi,S. and Kitagawa,G. (2008). {\it Information Criteria and Statistical Modeling.} {Springer-Verlag,  New York.}

\noindent  Martins-Filho,C., Mishra,S. and Ullah,A. (2008). A class of improved parametrically guided nonparametric regression estimators. {\it Econometric Rev.} {\bf 27} 542-573.

\noindent  Naito,K. (2002). Semiparametric regression with multiplicative adjustment. {\it Communications in Statistics, Theory and Methods} {\bf 31} 2289-2309.

\noindent  Naito,K. (2004). Semiparametric density estimation by local $L_2$-fitting. {\it Ann. Statist.} {\bf 32} 1162-1191.

\noindent  Opsomer,J.D. (2000). Asymptotic properties of backfitting estimators.{\it J. Mult. Anal.} {\bf 73}, 166--79.

\noindent  Opsomer,J.D. and Ruppert,D. (1997). Fitting a bivariate additive model by local polynomial regression.{\it Ann. Statist.} {\bf 25}, 186-211.

\noindent  O'Sullivan,F. (1986). A statistical perspective on ill-posed inverse problems.{\it Statist. Sci.} {\bf 1}, 505--27.(with discussion).

\noindent  Ruppert,D., Sheather,S.J. and Wand,M.P. (1995). An effective bandwidth selector for local least squares regression. {\it J. Amer. Statist. Assoc.} {\bf 90} 1257-1270.

\noindent  Ruppert,D., Wand,M.P. and Carroll,R.J. (2003).\ {\it Semiparametric Regression}.\ Cambridge University Press.

\noindent  Zhou,S., Shen,X. and Wolfe,D.A. (1998). Local asymptotics for regression splines and confidence regions. {\it Ann. Statist.} {\bf 26}(5):1760-1782.

\end{document}